\pdfoutput=1 
\documentclass[12pt]{amsart}
\usepackage{hyperref}
\usepackage{graphicx,amsmath,amsfonts,latexsym,amssymb,amsthm,color,epsfig}
\usepackage{latexsym,pictexwd,dcpic}

\usepackage{float}

  \newcommand{\field}[1]{\mathbb{#1}}
  \newcommand{\rbb}{\field{R}}

  \newcommand{\zbb}{\field{Z}}

\newcommand{\AC}{\mathcal{A}}

\newcommand{\DC}{{\mathcal D}}
\newcommand{\IC}{{\mathcal I}}
\newcommand{\FC}{{\mathcal F}}

\newcommand{\OC}{{\mathcal O}}
\newcommand{\UC}{{\mathcal U}}
\newcommand{\VC}{{\mathcal V}}

\newcommand{\dr}{{\mathrm d}}
\newcommand{\eps}{{\varepsilon}}
\newcommand{\gr}{{\mathrm g}}
\newcommand{\I}{\,\mathrm{i}\,}

\newcommand{\conesupp}{\mathrm{cone\,supp\,}}
\newcommand{\PsDO}{\psi\mathrm{DO}}
\newcommand{\im}{\mathrm{Im}}
\newcommand{\Op}{\mathrm{Op}}
\newcommand{\rank}{\mathrm{rank}\,}
\newcommand{\supp}{\mathrm{supp}\,}
\newcommand{\sgn}{\mathrm{sgn}\,}
\newcommand{\Tr}{\mathrm{Tr}}
\newcommand{\WF}{\mathrm{WF}\,}

  \newtheorem{theorem}{Theorem}[section]
  \newtheorem{definition}[theorem]{Definition}
  \newtheorem{lem}[theorem]{Lemma}
  \newtheorem{pro}[theorem]{Proposition}
  \newtheorem{cor}[theorem]{Corollary}
  \newtheorem{rem}[theorem]{Remark}
  \newtheorem{Assumption}[theorem]{Assumption}
  \newtheorem{exa}[theorem]{Example}

\numberwithin{equation}{section}

\begin{document}

\title[Semiclassical theory of discontinuous systems]
{The Semiclassical theory of discontinuous systems and ray-splitting billiards}

\author[D. Jakobson]{Dmitry Jakobson}
\address{Department of Mathematics and Statistics, McGill
University, 805 Sherbrooke Str. West, Montr\'eal QC H3A 0B9,
Ca\-na\-da.} \email{jakobson@math.mcgill.ca}

\author[Y. Safarov]{Yuri Safarov}
\address{
Department of Mathematics,
King's College London,
Strand, London, WC2R 2LS,
United Kingdom }
\email{yuri.safarov@kcl.ac.uk}

\author[A. Strohmaier]{Alexander Strohmaier}
\address{Department of Mathematical Sciences,  Loughborough University,  Loughborough, Leicestershire, LE11 3TU,
UK} \email{a.strohmaier@lboro.ac.uk}

\author[Y. CdV (Appendix)]{ with an Appendix by Yves Colin de Verdi\`ere}
\address{Institut Fourier, Grenoble University,  BP74,
 38402-St Martin d'H\`eres Cedex
  (France)}
\email{yves.colin-de-verdiere@ujf-grenoble.fr}


\date{\today}

\thanks{D.J.\ was partially supported by NSERC, FQRNT and Dawson Fellowship}

\begin{abstract}
We analyze the semiclassical limit of spectral theory on manifolds whose metrics have jump-like discontinuities. Such systems are quite different from manifolds with smooth Riemannian metrics because the semiclassical limit does not relate to a classical flow but rather to branching (ray-splitting) billiard dynamics. In order to describe this system we introduce a dynamical system on the space of functions on phase space. To identify the quantum dynamics in the semiclassical limit we compute the principal symbols of the Fourier integral operators associated to reflected and refracted geodesic rays and identify the relation between classical and quantum dynamics. In particular we prove a quantum ergodicity theorem for discontinuous systems. In order to do this we introduce a new notion of ergodicity for the ray-splitting dynamics.
\end{abstract}

\maketitle


\section{Introduction}

Many questions about spectra and eigenfunctions of elliptic operators are motivated by  Bohr's correspondence principle in quantum mechanics, asserting that a classical dynamical  system manifests itself in the semiclassical (as Planck's constant $h\to 0$) limit  of its quantization. When the quantum system is given by the Laplacian $\Delta$ on the Riemannian manifold $M$  (describing a quantum particle on $M$ in the absence of electric and magnetic fields) the corresponding classical system is the geodesic flow $G^t$ on $M$, so in the high energy limit  
eigenfunctions should reflect the properties of the geodesic flow.  

One of the most studied question concerns limits of eigenfunctions.  To an eigenfunction $\phi_j$ with 
$\Delta\phi_j = \lambda_j\phi_j$ one can associate a measure $\dr\mu_j$ on $M$ with the density 
$|\phi_j|^2$; its phase space counterpart is a distribution $\dr\omega_j$ 
on the unit cosphere bundle $S^*M$, projecting to $\dr\mu_j$. It can be 
defined as follows: given a smooth function 
$a$ on $S^*M$ (an observable), we choose a ``quantization,'' a pseudodifferential operator $A=\Op(a)$ 
of order zero with principal symbol $a$, and let 
\begin{equation}\label{meas:def} 
\langle a,\dr\omega_j\rangle = 
\int_M (A\phi_j)(x)\,\overline{\phi_j(x)}\,\dr x:=
\langle A\phi_j,\phi_j\rangle.   
\end{equation}
The definition of the measures $\dr\omega_j$ depends of course on the choice of the quantization map $a \mapsto\Op(a)$, but any two choices differ by an operator of a lower order, and that does not affect the asymptotic  behaviour of $\langle A\phi_j,\phi_j\rangle$.  The measures $\dr\omega_j$ are sometimes called 
{\em Wigner measures.} 

A natural problem is to study the set of weak* limit points of $\omega_j$-s; it follows from Egorov's  theorem  that any limit measure 
is invariant under the geodesic flow, but the limits are quite different  for manifolds with integrable and ergodic geodesic flows.  
If $M$ has completely integrable geodesic  
flow $G^t$, sequences of $\dr\omega_j$-s concentrate on Liouville tori in phase space satisfying the quantization condition.  

Let
\begin{equation*}\label{counting_function}
N(\lambda) :=\ \#\{\lambda_j<\lambda^2\}
\end{equation*}
be the the counting function of the Laplacian (as usual, we enumerate eigenvalues taking into account their multiplicities). Recall that, by the Weyl formula,
$$
N(\lambda)\ =\ \lambda^n (2\pi)^{-n}n^{-1}\mathrm{Vol}\,(S^*M)+o(\lambda^n)\,,\qquad\lambda\to+\infty\,.
$$

If $G^t$ is ergodic, the following fundamental result (sometimes called {\em quantum 
ergodicity} theorem) holds.

\begin{theorem}\label{thm:qerg} 
Let $M$ be a compact manifold with ergodic geodesic flow, and let $A$ be a zero order pseudodifferential operator with principal symbol $\sigma_A$.  Then  
\begin{equation*}\label{qerg:fla}
\sum_{\lambda_j\leq\lambda^2}\left|\langle A
\phi_j,\phi_j\rangle-\int_{S^*M} \sigma_A\,\dr\omega\right|^2\ =\ o\left(N(\lambda)\right)\,,\quad\lambda\to+\infty,
\end{equation*}
or, equivalently,
\begin{equation*}
\lim_{N\to\infty}\frac1N\,\sum_{j=1}^N\left|\langle A
\phi_j,\phi_j\rangle-\int_{S^*M} \sigma_A\,\dr\omega\right|^2\ =\ 0
\end{equation*}
where $\dr\omega$ is the normalized canonical measure on $S^*M$.
\end{theorem} 

The theorem shows that for a subsequence $\phi_{j_k}$ of eigenfunctions of the full density, $\dr\omega_{j_k}\to\dr\omega$, and after projecting to $M$ we find that $\phi_{j_k}^2\to 1$ (weak$^*$). In other words, almost all high energy eigenfunctions become equidistributed on the manifold and in phase  space.  
Various versions of Theorem \ref{thm:qerg} were proved in \cite{CV,HMR,Shn74,Shn93,Zelditch,Zelditch:96}
by Shnirelman, Zelditch, Colin de Verdi\`ere and Helffer--Martinez--Robert, as well as in other papers.  
Quantum ergodicity has been established for billiards in \cite{GL,ZZ}
 by G\'erard--Leichtnam and Zelditch--Zworski.  

Important further questions concern the rate of convergence in \eqref{qerg:fla}, quantum analogues 
of mixing and entropy.  Rudnick and Sarnak conjectured that on negatively curved manifolds, 
the conclusion of Theorem \ref{thm:qerg} holds {\em without averaging}, or equivalently $d\omega_j\to d\omega$ for {\em all} eigenfunctions; this is sometimes called {\em quantum unique ergodicity} (QUE).  This conjecture has been proved for some arithmetic hyperbolic 
manifolds by Lindenstrauss, with further progress by Soundararajan and Holowinsky.  
On the other hand, A. Hassell (\cite{Ha}) has shown that on the Bunimovich stadium billiard, there exist 
{\em exceptional} sequences of eigenfunctions concentrating on the ``bouncing ball'' orbits, so the 
analogue of the QUE conjecture does not hold for all billiards.  

The aim of this paper is to identify the correct semiclassical dynamics corresponding to quantum systems with discontinuities. More precisely, we are looking at manifolds (possibly with boundary) whose metrics are allowed to have jump discontinuities across codimension one hypersurfaces. Such manifolds model 
situations in physics where waves propagate in matter that consists of different layers  of materials. In the simplest case we would have two isotropic materials touching at a hypersurface. The metric in each layer is then given by $g_i=n_i(x)^2 g_e$, where $n_i(x)$ is the refraction index in layer $i$ and $g_e$ is the Euclidian metric on the tangent bundle. Wave propagation in these media is described by the wave equation
$$
 (\frac{\partial^2}{\partial t^2} + \Delta) \phi(x,t) =0, 
$$
where $\Delta$ is the Laplace operator with respect to the metric $g_i$
and transmissive boundary conditions are imposed on the solutions. The high energy limit of such a systems shows properties that do not remind of classical mechanics: singularities of solutions travel on geodesics until they hit the discontinuity. Then they are reflected and refracted according to the laws of geometrical optics. Consequently, there is no classical flow on phase space that describes the high energy limit.  Moreover, the naive generalization of Egorov's theorem fails in this situation. If $A$ is a pseudodifferential operator on the manifold supported away from the discontinuity then the quantum mechanical time-evolution $U(t) A U(-t)$ of $A$ will in general fail to be a pseudodifferential operator. Thus, on the algebraic level of observables the quantum-classical correspondence fails. We will show that after forming an average over the eigenstates  the quantum dynamics relates to a certain probabilistic dynamics that takes into account the different branches of geodesics emerging in this way. 
Our main result establishes a quantum ergodicity theorem in the case where this classical dynamics is ergodic.

Our proof relies on a precise symbolic calculus for Fourier integral operators associated with canonical transformations (see Section \ref{s:FIO})
and on a local Weyl law for such operators (Theorem \ref{t:local-fio}). We  construct a local parametrix for the wave kernel consisting of a sum of such Fourier integral operators (Section \ref{t:parametrix}) and apply to them the above results.

The usual proof of quantum ergodicity is based on the consideration of the positive operator obtained by squaring the average of the time-evolution of a pseudodifferential operator. Egorov's theorem plays an important role in this construction. Since it does not hold in our setting, the standard proof cannot be directly applied.
Instead we use the local Weyl law for an operator that is not necessarily positive but whose expectation value with respect to any eigenfunction is positive. We then apply the symbolic calculus to obtain an explicit formula for the leading asymptotic coefficient.

The proof of the main result is presented in full detail in Section \ref{s:ergodicity}. It should be mentioned that 
local Weyl laws of the type given in Theorem \ref{t:local-fio} are very powerful and have many applications. 
In particular, a less explicit but similar result was recently stated and used in
\cite{TothZelditchI} and \cite{TothZelditchII} to prove quantum ergodic restriction theorems.

Ray splitting not only occurs in the quantum systems we consider but 
also happens in situations described by systems of partial differential equations
and higher order equations. Moreover, ray-splitting occurs in a natural way in
quantum graphs. Our results carry over in a straightforward manner to these situations.


Ray-splitting billiards have been studied extensively in the Physics literature, see e.g. 
\cite{BYNK, BAGOP1, BAGOP2, BKS, COA, KKB,TS1, TS2} and references therein.  The emphasis has been on spectral statistics, trace formulae, eigenfunction localization 
(``scarring''), and the behaviour of periodic orbits.  
In the mathematical literature, the emphasis has been on the propagation of singularities 
\cite{Iv1} and spectral asymptotics \cite{Iv2, Sa2}.  

Quantum ergodicity has not previously been considered for ray-splitting (or branching) billiards.  
Our results fill an important gap in the semiclassical theory of such systems.  It is our hope 
that our results will serve as a motivation to further study a probabilistic dynamical system 
that we introduce.


\section{Setting: Manifolds with metric discontinuities along hypersurfaces}\label{s:setting}

Following Zelditch and Zworksi \cite{ZZ},
we say that $M$ is a compact manifold with piecewise smooth Lipschitz boundary if $M$ is a compact subset of a smooth
manifold $\tilde M$ such that there exists a finite collection of smooth functions $f_1,\ldots,f_\ell$ such that
\begin{enumerate}
 \item $df_j |_{f_j^{-1}(0)} \not= 0$,
 \item $M$ has Lipschitz boundary and $f_i^{-1}(0) \bigcap f_j^{-1}(0)$
  is an embedded submanifold of $\tilde M$.
 \item $M = \{ x \in \tilde M \,:\, \forall 1 \leq j \leq \ell : f_j(x) \geq 0\}$.
\end{enumerate}
A Riemannian metric on $M$ is assumed to be the restriction of a Riemannian metric on $\tilde M$.
We will denote the regular part of the boundary $\partial M$ of $M$ by $\partial M_{\mathrm{reg}}$
and the singular part by $\partial M_{\mathrm{sing}}$.

Suppose that $X$ is a compact manifold with piecewise smooth Lipschitz boundary and 
$Y \subset X$ is a co-dimensional one piecewise smooth closed hypersurface in $X$
such that $\partial Y \subset \partial X$. We will also assume that the completion
of $X \backslash Y$ with respect to the inherited uniform structure 
is again a manifold $M$ with piecewise smooth Lipschitz boundary. Thus, cutting $X$ open along
$Y$ results in $M$ and we obtain a part $N$ of the boundary $\partial M$ with a two-fold covering
map $N \to Y$.

\begin{figure}[H]
 \includegraphics*[width=5cm,height=4cm]{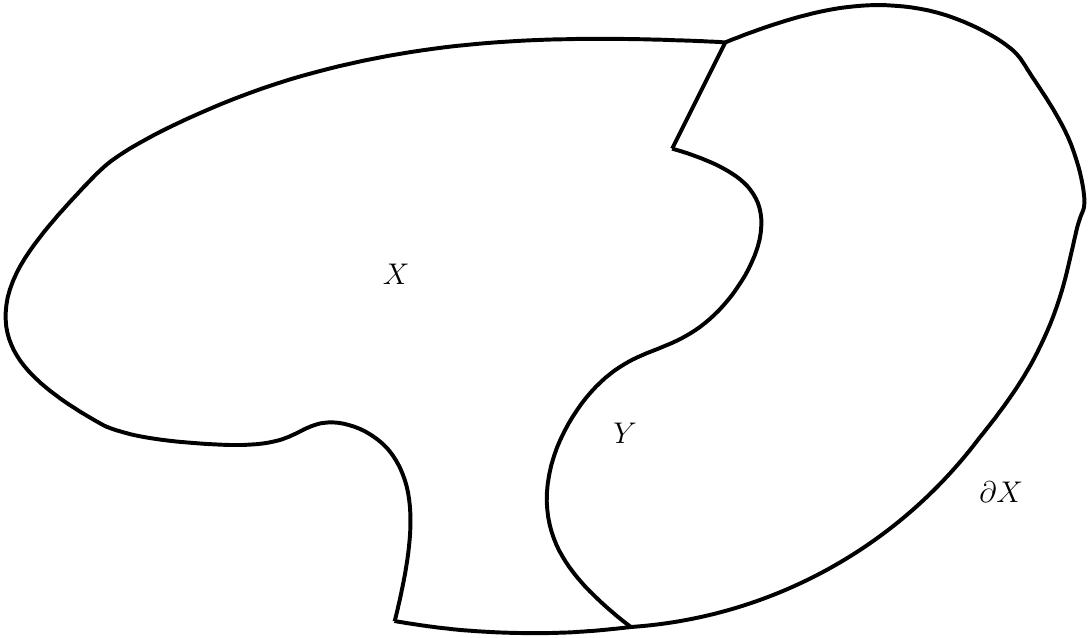} \hspace{1cm}
 \includegraphics*[width=5cm,height=4cm]{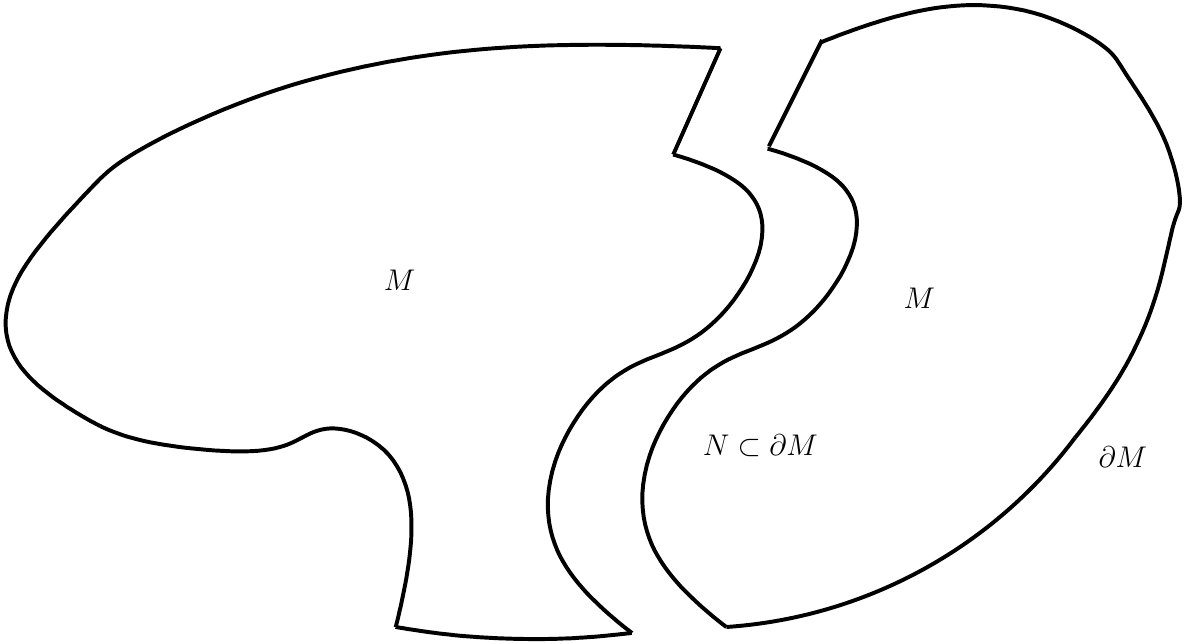}
 \caption{Domain in $\mathbb{R}^2$ with boundary} \label{fig1}
\end{figure}

\begin{exa}
In the simplest case $X$ is oriented, $Y$ is a closed hypersurface that separates
$X$ into two parts $X_1$ and $X_2$. Then $M$ will be the disjoint union of
$\overline{X_1}$ and $\overline{X_2}$, $N$ is the disjoint union of two copies of $Y$,
and the deck transformation 
simply interchanges these two copies of $Y$.
\end{exa}

\begin{figure}[H]
 \includegraphics*[width=5cm,height=2cm]{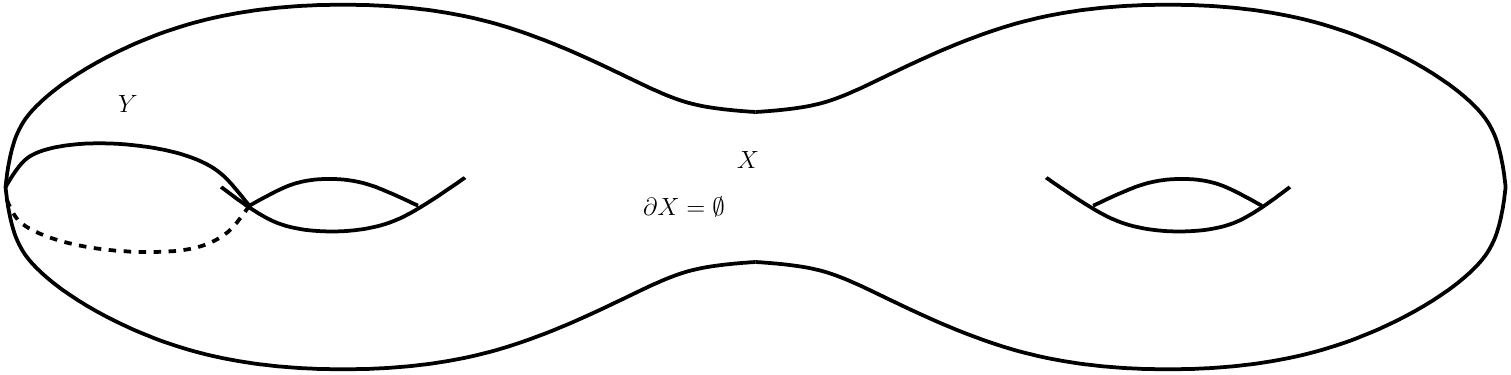} \hspace{1cm}
 \includegraphics*[width=5cm,height=2cm]{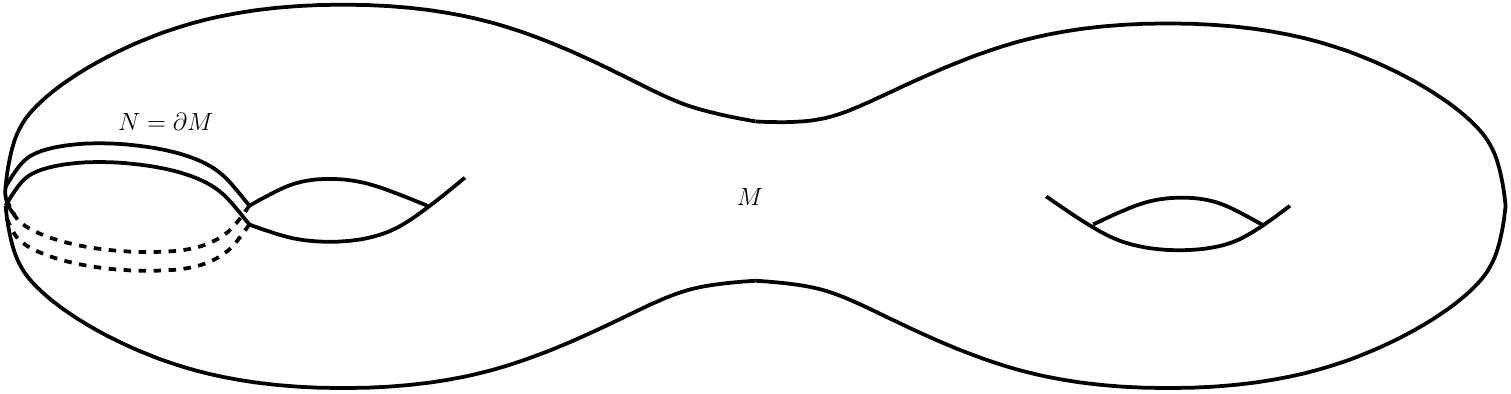}
 \caption{Closed manifold with metric discontinuity} \label{fig2}
\end{figure}

A smooth metric on $M$ defines a metric on
$X \backslash Y$. Since we do not require that the covering map $\phi$ be an isometry, this metric
can in general not be continued to a metric on $X$ but has a jump discontinuity at $Y$.
Manifolds $X$ with a metric of this form on $X \backslash Y$ can be thought of as
Riemannian manifolds with metric jump discontinuities at $Y$. Note that the construction
ensures that the metric can be continued smoothly up to $Y$ on either side of $Y$
in local coordinates (although the continuations from the left and from the right need not coincide).

The gluing construction defines a map $T^* M |_{N} \to T^*X |_Y$
which is again a two-fold covering map that lifts the original covering map.
Note that the deck transformation of this cover does not in general preserve the length of covectors.

Since $N$ has zero measure, functions in $L^p(M)$ can also be understood
as functions in $L^p(X)$. Moreover, functions in $C^\infty(M)$ can be understood
as functions in $L^\infty(X)$ which are smooth away from $Y$ and can have a jump discontinuity
along $Y$. For the sake of notational simplicity, in the following we will not distinguish
between the spaces $L^p(M)$ and $L^p(X)$ and understand $C^\infty(M)$ as a subspace
of $L^\infty(X)$.

We will assume that $D$ is either the Sobolev space
$H^1(X)$  or the space $H^1_0(X_{\mathrm{int}})$ of $H^1$-functions vanishing on $\partial X$. We define the Laplace
operator $\Delta$ on $X$ as the self-adjoint operator
defined by the Dirichlet quadratic form
\begin{equation}\label{q-form}
 q(\phi,\phi)=\int_{X \backslash Y} |\nabla\phi(x) |^2\,\mathbf g(x)\,\dr x
\end{equation}
with domain $D$, where $\mathbf g$ is the standard Riemannian density, $\mathbf g(x)\,\dr x$ is the volume element and
$|\nabla\phi(x)|$ is the Riemannian norm of the covector $\nabla\phi$. Contraction of the density $\mathbf{g}$ with the outward pointing 
unit normal vector field results in a density $\mathbf{h}$ on $N$. There then exists a positive function $b \in L^\infty(N) \cap C^\infty(N_\mathrm{reg})$ such that
$ {\mathbf h}(x) =  b(x) \widehat{\mathbf h}(x)$ for all $x \in N$, where $\widehat{\mathbf h}$ denotes the deck transformation
of the density $\mathbf{h}$ on $N$.

The domain of the unique self-adjoint operator generated by this quadratic form
coincides with
\begin{gather*}
\bigl\{f \in H^2(M) \bigcap H^1(X) \;:\; \\ \langle\widehat{\nabla f},n_N(x)\rangle = -  \langle\nabla f,n_N(x)\rangle \cdot b(x)
\textrm{ for all } x \in N,\   f|_{\partial X}=0 \bigr\}
\end{gather*}
in the case $D=H^1_0(X_{\mathrm{int}})$, and with
\begin{gather*}
\bigl\{f \in H^2(M) \bigcap H^1(X) \;:\; \\ \langle\widehat{\nabla f},n_N(x)\rangle = - \langle\nabla f, n_N(x)\rangle \cdot b(x)\textrm{ for all } x \in N,\  \langle\nabla  f,n_{\partial X}\rangle=0 \bigr\}
\end{gather*}
in the case $D=H^1(X)$.

In both formulae $n_N$ and $n_{\partial X}$ are the outward pointing unit normal vector fields at $N_{\mathrm{reg}}$ and $\partial X_{\mathrm{reg}}$
respectively, and $\widehat{\nabla f}$ is the deck transformation of $\nabla f$. Thus, in the former case the operator is subject to Dirichlet boundary conditions at $\partial X$
and transmissive boundary condition at $N$. In the latter case the operator is subject to Neumann boundary
conditions at $\partial X$ and transmissive boundary conditions at $N$. 

To be more specific and for notational purposes
we would like to describe the transmissive boundary condition in local coordinates near $Y_{\mathrm{reg}}$.
Suppose that $x_0 \in Y_{\mathrm{reg}}$ is a point.
Locally the normal bundle of $Y_{\mathrm{reg}}$
is trivial and, therefore, there exists a connected open neighbourhood $\OC$ of $x_0$ such 
that $\OC \backslash Y$ is the disjoint
union of two connected components $\OC_+$ and  $\OC_-$. 
After making such a choice, for each point $x \in Y \cap \mathcal{O}$ there are two points $x_+$ and $x_-$ in $N$
that correspond to $x$ under the covering $N \to Y$. These points $x_\pm$ can be thought of as the different limits in $M$
that sequences have when converging to $x$ in $\mathcal{O}_+$ and $\mathcal{O}_-$ respectively.
For any function $f(x)$ on $M$ we can then define its two boundary values $f_+(x):=f(x_+)$ and $f_-(x):=f(x_-)$.
The deck transformation group will act by interchanging the points $x_+$ and $x_-$. Similarly, different metrics
$\{g_+^{ij}(x)\}$ and $\{g_-^{ij}(x)\}$ on $T_x^*X$ are obtained by passing to the limit
in $\OC_+$ and  $\OC_-$ respectively. 
Let $g_\pm(x,\xi)$ be the corresponding quadratic forms, and let $\mathfrak{n}^\pm$ be the $g_\pm$-unit 
conormal vectors oriented into the $g_\pm$-sides.
In the same way let $h$ be the induced Riemannian
metric on $N_{\mathrm{reg}}$. Then $h_+(x)=h(x_+)$ and $h_-(x)=h(x_-)$ correspond to the different Riemannian metrics on 
$Y_{\mathrm{reg}}$ that are induced by the different limits from left and right hand side respectively.
The induced metric densities are then given by $\mathbf{(h}_\pm)(x) = \left( \mathrm{det}((h_\pm)_{ij}(x)) \right)^{1/2}$ and the function
$b$ is computed as $b_+(x)=b(x_+)=\mathbf{h}_+(x) \left( \mathbf{h}_-(x) \right)^{-1}$. By construction $b_-(x)=b(x_-)=(b_+(x))^{-1}$.
Elements $f$ in the domain of $\Delta$ then satisfy the transmissive boundary conditions
\begin{gather*}
 f_+(x) = f_-(x),\\
 b_+(x) \partial_n f_+(x) = -\partial_n f_-(x),
\end{gather*}
for all $x \in  Y_{\mathrm{reg}} \cap \OC$, where $\partial_n f_\pm$ is the exterior normal derivative
of $f$ at the point $x_\pm$.

Let $H^s(M)$ be the Sobolev space. Since these boundary conditions are elliptic, we have
$$
 \mathrm{dom}(\Delta^{s/2}) \subset H^s(M)
$$
for any $s>0$. Moreover, if $s>n/2+k$ then
$$
 \mathrm{dom}(\Delta^{s/2}) \subset C^k(M) \bigcap C(X).
$$

For technical reasons, it is more convenient to consider operators acting in the space of half-densities on $X$ rather than in the space of functions. Further on, if $\mathcal{H}(\cdot)$ 
is a function space, we shall denote by $\mathcal{H}(\cdot,\Omega^{1/2})$ the corresponding space of half-densities.

The operator ${\mathbf g}^{1/2}\,\Delta\,{\mathbf g}^{-1/2}$ with domain
$$
\{f \in L^2(X,\Omega^{1/2}) \,:\, {\mathbf g}^{-1/2}f\in \mathrm{dom}(\Delta) \}
$$
is said to be the {\it Laplacian in the space of half-densities.} Clearly, it is a self-adjoint operator 
whose spectrum coincides with the spectrum of $\Delta$. Moreover, $f$ is an eigenfunction of $\Delta$ if 
and only if the half-density ${\mathbf g}^{1/2} f$ is an eigenvector of ${\mathbf g}^{1/2}\,\Delta\,{\mathbf g}^{-1/2}$  
corresponding to the same eigenvalue.

We shall denote the Laplacian in the space of half-densities by the same letter $\Delta$ specifying, 
when necessary, in what space we consider the operator.

\section{Ray-splitting billiards}\label{s:billiards}

\subsection{Reflection and refraction}\label{s:billiards1}

As above let $\{g^{ij}(x)\}$ be the Riemannian metric on $T_x^*X$ and  
$g(x,\xi):=\sum_{i,j=1}^ng^{ij}(x)\,\xi_i\,\xi_j$. 
For local statements near $Y$ it will be enough to consider open sets 
$\OC$ as in the previous section and a choice $\OC_\pm$ as above so that we can use the above notation
$\{g_\pm^{ij}(x)\}$ for the different boundary values of the metric.
If $(y,\eta)\in T^*(X \backslash\partial M)$, where $y\in X \backslash\partial M$ and $\eta\in T^*_yX$, let denote by $(x^t(y,\eta),\xi^t(y,\eta))$ the Hamiltonian  trajectory generated by the Hamiltonian $\sqrt{g(x,\xi)}$ and starting at 
$(x^0(y,\eta),\xi^0(y,\eta)):=(y,\eta)$. Its projection $x^t$ onto $X$ is the geodesic emanating from the point $y$ in 
the direction $\eta$. Clearly, $x^t$ and $\xi^t$ are positively homogeneous functions of $\eta$ of degree 0 and 1 respectively. Furthermore, if  $(y,\tilde\eta)\in S^*M$ then $(x^t(y,\tilde\eta),\xi^t(y,\tilde\eta))\in S^*M$ and the geodesic $x^t$ is parametrised by its length. 

The Hamiltonian  trajectory $(x^t,\xi^t)$ is well defined until the geodesic $x^t$ hits the boundary $\partial M$, i.e., either the set
$Y$ or the boundary of $X$. In the former case, according to the laws of geometrical optics, it splits into two geodesics. 
One of them is obtained by reflection, and the other is the refracted trajectory which goes through $Y$ but changes its direction. In the latter case there is only the reflected trajectory. 
More precisely, there are the following possibilities.

Assume, for the sake of definiteness, that the trajectory 
$$
\gamma_+(t)=(x^t(y,\eta),\xi^t(y,\eta))
$$ 
approaches $Y$ from the $g_+$-side and meets $Y_{\mathrm{reg}}$ at the time $t^*$. Let 
$$
\lim_{t\to t^*-0}(x^t(y,\eta),\xi^t(y,\eta))\ :=\ (x^*,\xi^*)\,,
$$ 
so that $x^*\in Y_{\mathrm{reg}}$ and $g_+(x^*,\xi^*)=g(y,\eta)$. Denote by $\xi_Y^\pm$ the $g_\pm$-orthogonal projections 
of $\xi^*$ onto the cotangent space $T_{x^*}^*Y$.

Clearly, $\xi^*=\xi_Y^+-\tau_+\,\mathfrak{n}^+$ where 
$$
\tau_+\ =\ \sqrt{g_+(x^*,\xi^*)-g_+(x^*,\xi_Y^+)}\,.
$$ 
By definition, 
the {\it reflected} trajectory is the Hamiltonian trajectory $\tilde\gamma_+(t)$ originating from the point 
$(x^*,\xi_Y^++\tau_+\,\mathfrak{n}^+)$ at the time $t^*$ and going into the $g_+$-side of $X$. 
If $\tau_+=0$ then the geodesic $x^t$ hits $Y_{\mathrm{reg}}$ at zero angle. In this case
the reflected trajectory is not well defined.

Assume that $g_-(x^*,\xi_Y^-)<g_+(x^*,\xi^*)$ and denote 
$$
\tau_-\ =\ \sqrt{g_+(x^*,\xi^*)-g_-(x^*,\xi_Y^-)}\,. 
$$
Obviously,
$g_-(x^*,\xi_Y^-\pm\tau_-\mathfrak{n}^-))=g_+(x^*,\xi^*)$. The Hamiltonian trajectory $\tilde\gamma_-(t)$ originating from the 
point $(x^*,\xi_Y^-+\tau_- \mathfrak{n}^-)$ at the time $t^*$ and going into the 
$g_-$-side is called the {\it refracted} trajectory. Note that in this case $\tilde\gamma_-(t)$ 
is the reflection of the trajectory $\gamma_-(t)$ coming from the $g_-$-side to the point $(x^*,\xi_Y^--\tau_- \mathfrak{n}^-)$.
The corresponding refracted trajectory coincides 
with $\tilde\gamma_+(t)$, so that  $\gamma_+(t)$ and $\gamma_-(t)$ have the same pair of reflected and refracted trajectories.

If $g_-(x^*,\xi_Y^-)>g_+(x^*,\xi^*)$ then there is no refraction. In this case one says that $(x^*,\xi^*)$ is a point of 
{\it total reflection.} If $g_-(x^*,\xi_Y^-)=g_+(x^*,\xi^*)$ then the angle of refraction is zero and 
the refracted trajectory may not be well defined.

If the geodesic hits a point $x^*\in\partial X_{\mathrm{reg}}$ then $(x^t,\xi^t)$ is reflected in the same way as above. Namely, the reflected trajectory is the Hamiltonian trajectory originating from $(x^*,\xi^*_{\partial X}-\xi^*_n)$ where $\xi^*_{\partial X}$ and $\xi^*_n$ are tangential and normal components of $\xi^*$.

\subsection{Billiard trajectories}\label{s:billiards2}
The trajectory obtained by consecutive reflections and/or refractions is called a billiard trajectory. In general, there are infinitely many billiard trajectories originating from a given point 
$(y,\eta)\in T^*(X \backslash\partial M)$; moreover, the set of these trajectories is typically uncountable. 
We shall denote them by $(x^t_\kappa,\xi^t_\kappa)$, where $\kappa$ is an index specifying 
the type of trajectory (see below). Each billiard trajectory $(x^t_\kappa,\xi^t_\kappa)$ 
consists of a collection of geodesic segments $(x^t_{\kappa,j},\xi^t_{\kappa,j})$, which are joined at $Y$ and $\partial X$.

Following \cite{Sa1}, we shall suppose that $\kappa$ is a ternary fraction 
$$0.\kappa_1\kappa_2\ldots,$$ where $\kappa_m=0$ or $\kappa=2$ for all $m$, so 
that $\kappa$ is a point of the Cantor set in $[0,1]$. More precisely, we 
say that the trajectory has type $\kappa$ if the following is true.
\begin{itemize}
\item If the $(m+1)$st segment of the trajectory is obtained by reflection and there exists the 
corresponding refracted ray then $\kappa_m=0$.
\item  if the $(m+1)$st segment of the trajectory is obtained by refraction then $\kappa_m=2$,
\item  If the trajectory has only $m$ segments or the $m$th segment ends at $\partial X$ or at a point of total reflection then either $\kappa_m=0$ or $\kappa_m=2$.
\end{itemize}
The last condition implies that a billiard trajectory may have different types $\kappa$. Roughly speaking, the equality $\kappa_m=2$ means that $(m+1)$st segment is obtained by refraction whenever it is possible.

\begin{figure}[H]
 \includegraphics*[width=10cm,height=5cm]{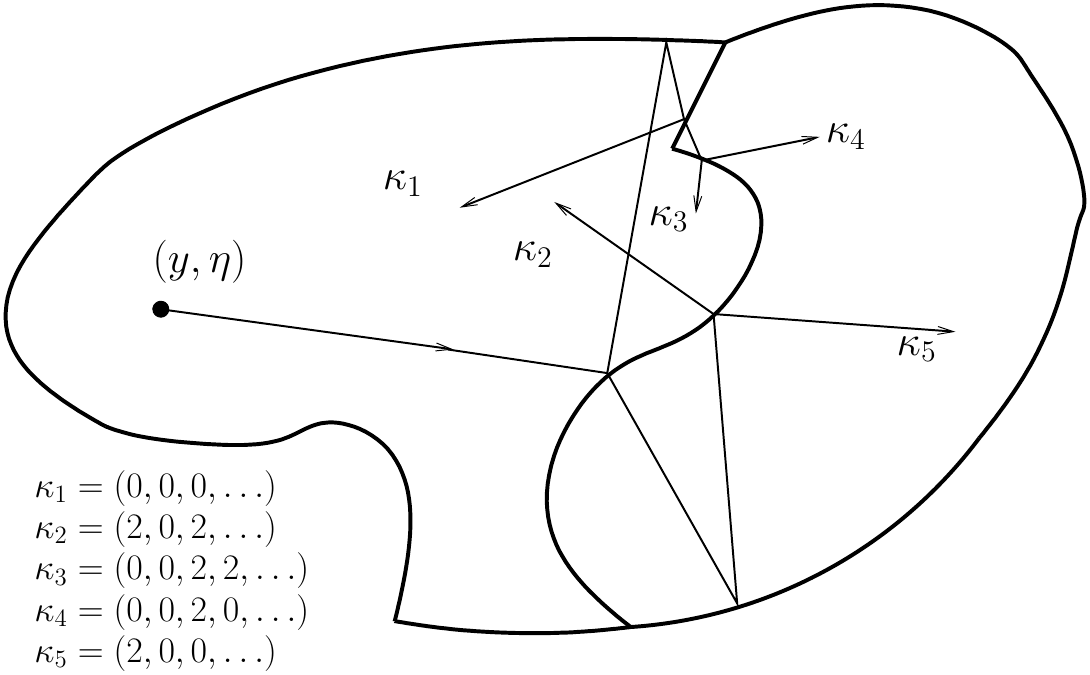} \hspace{1cm}
 \caption{Ray-splitting trajectories} \label{fig3}
\end{figure}

\subsection{Dead-end and grazing trajectories}\label{s:billiards3}
A billiard trajectory is not well-defined if
\begin{itemize}
\item the trajectory hits $\partial M$ infinitely many times in a finite time,
\item or the angle of incidence or the angle of refraction is equal to zero,
\item or the trajectory hits a point in $\partial M_{\mathrm{sing}}$.
\end{itemize}
Trajectories of the first type are called {\it dead-end}, trajectories of the second type are said to 
be {\it grazing}, and we call trajectories of the third kind {\it singular}. 
Let $O_\dr$, $O_\gr$, and $O_\mathrm{s}$ be the sets of starting points of the dead-end, 
grazing and singular trajectories, respectively. Clearly, $O_\dr$, $O_\gr$, and $O_\mathrm{s}$ are conic subsets of $T^*(X\setminus\partial M)$. Throughout the paper we shall suppose that the following assumption holds.

\begin{Assumption}\label{a:bad-points}
The conic set $O_\dr\bigcup O_\gr \bigcup O_\mathrm{s}$ has measure zero in the cotangent bundle.
\end{Assumption}

This assumption means the set of starting points of ``bad'' trajectories is of measure zero or, in other words, that the billiard trajectories $(x^t_\kappa(y,\eta),\xi^t_\kappa(y,\eta))$ are well-defined for all $\kappa$, all $t\geq0$ and almost all $(y,\eta)\in T^*X$.

\begin{rem}
One can easily show that $O_\gr\bigcup O_\mathrm{s}$ is a set  of measure zero (see, for instance \cite{Sa1}, \cite{SV1} or \cite{SV2}). However, there are reasons to believe that $O_\dr\bigcap S^*M$ may have a positive measure. 
In \cite{SV1} the authors constructed such an example for a similar branching billiard.
\end{rem}

Let $\OC_{T}$ be the set of points $(y,\eta)\in T^*(X\setminus\partial M)$ such that all the billiard trajectories $(x^t_\kappa(y,\eta),\xi^t_\kappa(y,\eta))$ are well defined for $t\in[0,T]$ and $x^T_\kappa(y,\eta)\not\in\partial M$. Clearly,  $\OC_T$ is an open conic subset of $T^*(X\setminus\partial M)$. Assumption \ref{a:bad-points} implies that their intersection $\OC_\infty:=\bigcap_{T>0}\OC_T$ is a set of full measure in $T^*X$.

Note that the  mapping 
$\Phi^t_\kappa:(y,\eta)\mapsto(x^t_\kappa,\xi^t_\kappa)$ defined on $\OC_T$ is a homogeneous canonical transformation in the sense in symplectic geometry
for each fixed $t\in[t,T]$ and $\kappa$. It preserves the canonical symplectic 1-form 
$\xi\cdot{\mathrm d}x$ on $T^*X$ and 
the standard measures on $T^*X$ and $S^*M$ (see, for instance, \cite{Sa1} or \cite{SV2}).

\section{Fourier integral operators}\label{s:FIO}

Further on we shall abbreviate the words `Fourier integral operator' and `pseudodifferential operator'  to FIO and $\PsDO$, respectively. We shall always be assuming that their symbols and amplitudes belong to H\"ormander's classes $S^m_{\mathrm{phg}}$ (see, for instance, \cite[Chapter 18]{H2}). Recall that the conic support $\,\conesupp p\,$ of a function $\,p\in S^m_{\mathrm{phg}}\,$ is defined as the closure of the union $\bigcup_j\supp p_j$, where $p_j$ are the positively homogeneous functions appearing in the asymptotic expansion $p\sim\sum_j{p_j}$.

\subsection{Definition}\label{FIO:def}
Let $\Phi:(y,\eta)\mapsto(x^*(y,\eta),\xi^*(y,\eta))$ be a smooth homogeneous canonical transformation in $T^*M$ defined on an open conic set $\DC(\Phi)\subset T^*M$, and let $V:C_0^\infty(M,\Omega^{1/2})\mapsto C^\infty(M,\Omega^{1/2})$ be an operator with Schwartz kernel $\VC(x,y)$ (that is, $Vu(x)=\langle\VC(x,\cdot),u(\cdot)\rangle$).

The operator $V$
is said to be a FIO of order $m$ associated with $\Phi$ if  $\VC(x,y)$ can be written as an oscillatory integral of the form
\begin{equation}\label{integral}
(2\pi)^{-n}\int_{T^*_yM} e^{i\varphi(x,y,\eta)}p(y,\eta)\,\left|\det\varphi_{x\eta}(x,y,\eta)\right|^{1/2}\,
\varsigma(x,y,\eta)\,d\eta
\end{equation}
modulo a half-density from $C^\infty(M\times M,\Omega^{1/2})$. Here $\varsigma$, $\varphi$ and $p$ are smooth functions on $M\times T^*M$ satisfying the following conditions.
\begin{enumerate}
\item[{\bf(a$_1$)}]
$\varsigma$ is an arbitrary cut-off function, which is positively homogeneous of degree 0 for large $\eta$, is identically equal to 1 in a small neighbourhood of the set $\{x=x^*(y,\eta)\}$ and vanishes outside another small neighbourhood of the set $\{x=x^*(y,\eta)\}$.
\item[{\bf(a$_2$)}]
$p$ is an amplitude from a class $S^m_{\mathrm{phg}}$ with $\conesupp p\subset\DC(\Phi)$;
\item[{\bf(a$_3$)}]
$\varphi$ is positively homogeneous in $\eta$ of degree 1 with  $\im\,\varphi\geq0$.
\item[{\bf(a$_4$)}]
$\varphi(x,y,\eta)=(x-x^*)\cdot\xi^*+O\left(|x-x^*|^2\right)$ as $x\to x^*$.
\item[{\bf(a$_5$)}]
$\det\varphi_{x\eta}(x,y,\eta)\ne0$ for all $(x,y,\eta)\in\supp\varsigma$ and $x=x^*$ is the only solution of
the equation $\varphi_\eta(x,y,\eta)=0$ on $\supp\varsigma$.
\end{enumerate}

Note that 
\begin{itemize}
\item 
$\varphi_\eta(x^*,y,\eta)=0$ because $\Phi$ preserves the 1-form $x\cdot\dr\xi$. Therefore the condition {\bf(a$_5$)} is fulfilled whenever $\det\varphi_{x\eta}(x^*,y,\eta)\ne0$ and $\supp\varsigma$ is small enough.
\item
The right hand side of \eqref{integral} behaves as a half-density with respect to $x$ and $y$ and, consequently, the corresponding operator acts in the space of half-densities;
\end{itemize}

\begin{rem}\label{r:fio-def1}
The above definition of a FIO was introduced in \cite{LSV} (see also \cite[Chapter 2]{SV2}). It is equivalent to the traditional one, which is given in terms of local real-valued phase functions parametrizing the Lagrangian manifold 
\begin{equation}\label{lagr-manifold}
\{(y,\eta;x,\xi)\in\DC(\Phi)\times T^*M\,:\,(x,\xi)=(x^*(y,\eta),\xi^*(y,\eta))\}
\end{equation}
(see, for example, \cite{H1} or \cite{Tr}).
\end{rem}

\begin{rem}\label{r:fio-def2}
By \cite[Theorem 1.8]{LSV}, the Schwartz kernel of a FIO can be represented by an integral of the form \eqref{integral} with any phase function $\varphi$ and 
cut-off function $\varsigma$ satisfying the above conditions.
\end{rem}

\begin{rem}\label{r:fio-def3}
One can define a FIO using \eqref{integral} with an amplitude $\tilde p(x,y,\eta)\in S^m_{\mathrm{phg}}$ depending on $x$ instead of  $p(y,\eta)$. 
These two definitions are equivalent. Indeed, since $(x-x^*)e^{i\varphi}=B\,\nabla_\eta\,e^{i\varphi}$ with some smooth matrix-function $B$, one can always
remove the dependence on $x$ by expanding $\tilde p$ into Taylor's series at the point $x=x^*$, replacing $(x-x^*)e^{i\varphi}$ with $B\,\nabla_\eta\,e^{i\varphi}$ 
and integrating by parts. In particular, this procedure shows that \eqref{integral} with an $x$-dependent amplitude $\tilde p(x,y,\eta)$ defines an infinitely 
smooth half-density whenever $p\equiv0$ in a conic neighbourhood of the set $\{x=x^*\}$.
\end{rem}

One can find all homogeneous terms in the expansion of the amplitude $p(y,\eta)$ by analysing asymptotic behaviour of the Fourier transforms of localizations 
of the distribution \eqref{integral}. This implies that $p(y,\eta)$ is determined modulo a rapidly decreasing function by the FIO and the phase function $\varphi$. 
It is not difficult to show that the conic support $\conesupp p$ does not depend on the choice of $\varphi$ and is determined only by the FIO $V$ itself 
(see, for instance,  \cite[Section 2.7.4]{SV2}). We shall denote it by $\conesupp V$.

\begin{rem}\label{r:fio-def4}
By \cite[Corollary 2.4.5]{SV2}, if a phase function satisfies {\bf(a$_3$)}, {\bf(a$_4$)} and
the matrix $\mathrm{Im} \varphi_{xx}(x^*,y,\eta)$ is positive definite then $\,\det\varphi_{x\eta}(x^*,y,\eta)\ne0$. One can deduce from this that
the set of phase functions $\varphi$ satisfying {\bf(a$_3$)}--{\bf(a$_5$)} is connected and simply connected.
\end{rem}

\subsection{The Keller--Maslov bundle}

Let $\mathcal{D}_{\mathbb{Z}}(\Phi)$ be the $\mathbb{Z}$-principal bundle over $\mathcal{D}(\Phi)$ on which the multivalued function
$\arg{\det}^2\varphi_{x\eta}(x^*,y,\eta)$ becomes a single valued continuous function of $y,\eta$ depending continuously on $\varphi$. The fibre at the point $(y,\eta)$ can be thought of as the set of equivalence classes of pairs $(\varphi,a)$,
where $\varphi$ is a phase function satisfying {\bf(a$_3$)}--{\bf(a$_5$)}, $a$ is an integer, and the equivalence relation is
$(\varphi,a) \sim (\tilde \varphi, \tilde a)$ iff  
$$
 2 \pi ( a - \tilde a) - \arg{\det}^2\varphi_{x\eta}(x^*,y,\eta) + \arg{\det}^2\tilde \varphi_{x\eta}(x^*,y,\eta) \in (-\pi, \pi],
$$
where the branch of the argument in the right hand side is chosen to be continuous along any path in the set of phase functions satisfying {\bf(a$_3$)}--{\bf(a$_5$)}.
Then $\arg{\det}^2\varphi_{x\eta}(x^*,y,\eta)$ can be defined as a continuous single valued function on the principal bundle
$\mathcal{D}_{\mathbb{Z}}(\Phi)$ determined by
$$
 \arg{\det}^2\varphi_{x\eta}(x^*,y,\eta,[(\varphi,a)]) \in (-\pi + a, \pi + a ].
$$
Locally real phase functions provide local trivializations of $\mathcal{D}_{\mathbb{Z}}(\Phi)$ determining the  topology on the total space.

In the following we will often suppress the argument $[(\varphi,a)]$ and think of the function $\arg{\det}^2\varphi_{x\eta}(x^*,y)$
as a multivalued function on $\mathcal{D}(\Phi)$ understood as a continuous single valued function on $\mathcal{D}_{\mathbb{Z}}(\Phi)$.
Factoring out $4 \,\mathbb{Z} \subset \mathbb{Z}$ one obtains a $\mathbb{Z}_4$-principal bundle which we denote
by $\mathcal{D}_{\mathbb{Z}_4}(\Phi)$.

The complex line bundle associated with the representation 
$$
\mathbb{Z}_4 \to \mathbb{C}^{\times}, \quad a \mapsto i^{-a}
$$ 
is called the Keller--Maslov line bundle. Our definition here is equivalent to the one given in the literature 
(see, for instance, \cite{H1} and \cite{Tr}) because 
if $\varphi_{x\eta}(x^*,y,\eta)$ and $\tilde \varphi_{x\eta}(x^*,y,\eta)$ are real then $(\varphi,a) \sim (\tilde \varphi, \tilde a)$
is equivalent to 
$$
 a - \tilde a = -\frac{1}{2} \left( \sgn\varphi_{\eta \eta}(x^*,y,\eta) - \sgn\tilde\varphi_{\eta \eta}(x^*,y,\eta) \right).
$$
Sections of the Keller--Maslov line bundle can be understood as functions $f(y,\eta,\mathfrak{a})$ on 
$\mathcal{D}_{\mathbb{Z}_4}(\Phi)$ satisfying the equivariance condition
\begin{gather} \label{equivariance}
 f(y,\eta,\mathfrak{a}+n)=i^{-n} f(y,\eta,\mathfrak{a}),
\end{gather}
where $\mathfrak{a}=[(\phi,a)]$ denotes the variable in the fibre on which the $\mathbb{Z}$-action is defined by
$[(\phi,a)]+n = [(\phi,a+n)]$.
For the purpose of this article we will think of them in this way.

The situation simplifies when the bundle $\mathcal{D}_{\mathbb{Z}}(\Phi)$ is topologically trivial. This is equivalent to the existence a branch of $\arg(\det^2\varphi_{x\eta}(x^*,y,\eta))$ which is continuous on the set $\DC(\Phi)\,$. Clearly, such a branch exists whenever $\DC(\Phi)$ is simply connected.

\subsection{The index function $\Theta_\Phi$}

Let $C=C_1+iC_2$ be a symmetric $n\times n$-matrix with a nonnegative (in the sense of operator theory) real part $C_1$, and let $\Pi_C$ be the orthogonal 
projection on $\ker C$. We shall denote $\det_+C=\det(C+\Pi_C)$. Furthermore, we define the function $\arg{\det}_+C$ in such a way that it is continuous with 
respect to $C$ on the set of matrices with a fixed kernel and is equal to zero when $C_2=0$. In particular, if $C_1=0$ then $\arg\det_+C=\frac\pi2\,\sgn C_2$ where $\sgn C_2$ 
is the signature of $C_2$ (see \cite[Section 3.4]{H2}).

The following is \cite[Proposition 2.3]{LSV}.

\begin{pro}\label{p:index}
The  function 
\begin{multline}\label{index1}
\Theta_\Phi(y,\eta)\ =\ \frac1{2\pi}\,\arg{\det}^2\varphi_{x\eta}(x^*,y,\eta)\\
-\frac1\pi\,\arg{\det}_+(\varphi_{\eta\eta}(x^*,y,\eta)/i)\;+\;\frac12\,\rank x^*_\eta(y,\eta)
\end{multline}
is integer-valued and,
as a function on $\mathcal{D}_{\mathbb{Z}}(\Phi)$, does not depend on $\varphi$  and local coordinates. 
The function $\Theta_\Phi$ is continuous along any path on which $\rank x^*_\eta$ is constant.
\end{pro}

\begin{definition}\label{d:index}
If $\gamma: [a,b] \to \DC(\Phi)$ is a path in $\DC(\Phi)$ let $\tilde \gamma:[a,b] \to \DC_{\mathbb{Z}}(\Phi)$ be any continuous lift.  
$\Theta_\Phi(\tilde\gamma(b)) - \Theta_\Phi(\tilde\gamma(a))$ is independent of the lift and called the Maslov index of $\gamma$.
\end{definition}

The above definition was introduced in \cite{LSV} (the idea goes back to \cite{Ar}). If the path is closed then it
coincides with the Maslov index defined in \cite{H2} via a \v{C}ech cohomology class associated with parametrizations of the Lagrangian manifold \eqref{lagr-manifold} by 
families of local real-valued phase functions.  The index function $\Theta_\Phi$ allows one to extend this notion to non-closed paths.

\subsection{The principal symbol of a FIO}\label{FIO:symbol}

Choosing $\varsigma$ with a sufficiently small support, we can rewrite \eqref{integral} in the form 
\begin{equation}\label{integral1}
(2\pi)^{-n}\int_{T^*_yM} e^{i\varphi(x,y,\eta)}q(y,\eta)\,\left(\mathrm{det}^2\,\varphi_{x\eta}(x,y,\eta)\right)^{1/4}\,
\varsigma(x,y,\eta)\,d\eta\,,
\end{equation}
where $q$ is the section of the Keller--Maslov bundle obtained from the amplitude 
$\tilde q(x,y,\eta)=p(y,\eta)\,e^{-\frac i4(\arg({\det}^2\varphi_{x\eta}(x,y,\eta))}$ by the procedure described in Remark \ref{r:fio-def3}. Clearly, $q$ belongs to the same class $S^m_{\mathrm{phg}}$ and 
\begin{equation}\label{symbol-1}
q_0(y,\eta)\ =\ p_0(y,\eta)\,e^{-\frac i4(\arg({\det}^2\varphi_{x\eta}(x^*,y,\eta))}\,,
\end{equation}
where $q_0$ and $p_0$ are the leading homogeneous terms  of the amplitudes $q$ and $p$.

By construction, $\left(\mathrm{det}^2\,\varphi_{x\eta}(x,y,\eta)\right)^{1/4}$ is well defined for 
$x$ sufficiently closed to $x^*$ as a continuous function on $\mathcal{D}_{\mathbb{Z}_4}(\Phi)$ satisfying the equivariance condition
\begin{gather} \label{equivariance2}
 f(y,\eta,\mathfrak{a} + n)=i^{n} f(y,\eta,\mathfrak{a}),
\end{gather}
where $\mathfrak{a}$ is a variable for the fibre of $\mathcal{D}_{\mathbb{Z}_4}(\Phi)$.
Therefore it is a section in the dual of the Keller--Maslov bundle. Since the product of 
$q(y,\eta)\left(\mathrm{det}^2\,\varphi_{x\eta}(x,y,\eta)\right)^{1/4}$ is single valued, 
$q(y,\eta)$ is a section of the Keller--Maslov line bundle. We shall call it a {\it full symbol} of the corresponding FIO.

The following is \cite[Theorem 2.7.11]{SV2}.

\begin{theorem}\label{t:symbol}
Let $\,V\,$ be a FIO 
whose Schwartz kernel is given by \eqref{integral1}. Then the leading homogeneous term $q_0$ of the amplitude $\,q\,$
is uniquely determined by the operator $\,V\,$. 
\end{theorem}

\begin{rem}\label{r:proof2}
By Theorem {\rm\ref{t:symbol}}, the leading homogeneous term $q_0$ does not depend on the choice of local 
coordinates and does not change when 
we change the phase function $\varphi$ in the representation \eqref{integral1}.
\end{rem}

\begin{rem}\label{r:proof3}
Theorem {\rm\ref{t:symbol}} was proved in \cite{SV1} only for canonical transformations associated with billiards. 
However, the same proof works in the general case.
\end{rem}

In the following we will denote the leading homogeneous term $q_0$ of the symbol of a FIO $V$ by $\sigma_V$ and 
think of it as a multivalued function
on $\mathcal{D}(\Phi)$ or as a single valued function on $\mathcal{D}_{\mathbb{Z}_4}(\Phi)$ respectively, 
that satisfies (\ref{equivariance}).
Note that the product $i^{\Theta_\Phi}\sigma_V$ is single valued by construction as
\begin{equation}\label{sigma-theta}
 i^{\Theta_\Phi}\sigma_V \ =\ i^{(\rank x^*_\eta)/2-(\arg{\det}_+(\varphi_{\eta\eta}/i)\pi}\,p_0
\end{equation}
where $p_0$ is the leading homogeneous term of the amplitude $p$ from \eqref{integral} and 
$\,\arg{\det}_+(\varphi_{\eta\eta}/i)\,$ is evaluated at $x=x^*(y,\eta)$.

\begin{exa}\label{e:symbol}
If $\Phi$ is the identity transformation then $\varphi_{x\eta}(x^*,y,\eta)\equiv I$ and the corresponding 
FIO $\,V\,$ is a $\PsDO$ (see, for example, \cite[Theorem 19.1]{Sh}). 
If we put  $\arg(\det^2\varphi_{x\eta})\equiv0$ then the principal symbol of $\,V\,$ 
(in the sense of the theory of pseudodifferential operators) coincides with the leading homogeneous term $q_0$.
\end{exa}

If the bundle $\mathcal{D}_{\mathbb{Z}_4}(\Phi)$ is trivial 
we can globally fix a branch of $$\arg({\det}^2\varphi_{x\eta}(x,y,\eta)$$ for a fixed phase function.
Note, that in general there will be no preferred branch. In case 
$V$ is a $\PsDO$ we shall however suppose that $\arg(\det^2\varphi_{x\eta})\equiv 0$, 
so that in this case $\sigma_V$ coincides with the traditional principal symbol.

\subsection{Symbolic calculus for FIOs}\label{calculus}

In what follows, in order to avoid `boundary effects' when considering compositions of various FIOs and $\PsDO$s on $M$, we shall have to assume that supports of their 
Schwartz kernels are separated from the boundary. Namely, we shall deal the following classes of operators.
\begin{itemize}
\item $\AC'_0$ is the class of operators $V$ whose Schwartz kernels $\VC(x,y)$ vanish in a neighbourhood of the set $\partial M\times M$.
\item $\AC''_0$ is the class of operators $V$ whose Schwartz kernels $\VC(x,y)$ vanish in a neighbourhood of the set $M\times\partial M$.
\item $\AC_0:=\AC'_0\bigcap\AC''_0$.
\item $\AC$ is the class of operators that can be written in the form $cI+A_0$ where $A_0\in\AC_0$. 
\end{itemize}

The following results are in principle well-known. However, we need explicit formulae for the principal symbols, 
which are not obvious (partly, due to the fact that there are many 
possible definitions of the symbol). Therefore we have included a simple direct proof of Theorem 
\ref{t:composition} in Appendix \ref{s:a}.
In the rest of this section $\mathfrak{b} = \mathfrak{b}(y,\eta,\mathfrak{a}_1,\mathfrak{a}_2)$ or $\mathfrak{b} = \mathfrak{b}(y,\eta,\mathfrak{a})$ will denote
a map between the fibres of the $\mathbb{Z}_4$-bundles that depends only on the involved canonical transformations.
A construction of this map is implicitly contained in the proofs of the statements but is not needed for our purposes.

\begin{theorem}\label{t:composition}
Let $V_j$ be FIOs of order $m_j$ associated with canonical transformations $\Phi_j$, where $j=1,2$. 
Assume that either $V_1\in\AC'_0$ or $V_2\in\AC'_0$. 
Then the composition $V_2^*V_1$ is a FIO of order $m_1+m_2$ associated with the canonical 
transformation $\Phi=\Phi_2^{-1}\circ\Phi_1$ with principal symbol equal to
$\,\sigma_{{V_2}^*{V_1}} (y,\eta,\mathfrak{b}) = 
\sigma_{V_1}(y,\eta,\mathfrak{a}_1)\,\overline{\sigma_{V_2}(\Phi(y,\eta),\mathfrak{a}_2)}\,$  such that
$$
\conesupp(V_2^*V_1)\subset\left(\conesupp V_1\,\bigcap\,\Phi^{-1}(\conesupp V_2)\right).
$$
\end{theorem}

Recall that the inner product in the space of half-densities $L^2(M,\Omega^{1/2})$ is invariantly defined, and so are the adjoint operators. Taking $V_1=I$ in Theorem \ref{t:composition}, we obtain

\begin{cor}\label{c:adjoint}
Let $V$ be a FIO of order $m$ associated with a canonical transformation $\Phi$. If $V\in\AC'_0$ then the adjoint operator $V^*$ is a FIO of order $m$ associated with 
the inverse transformation $\Phi^{-1}$  with principal symbol equal to 
$\,\sigma_{V^*}(y,\eta,\mathfrak{b})=\overline{\sigma_V(\Phi^{-1}(y,\eta),\mathfrak{a})}\,$
such that
$$
\conesupp V^*\subset\Phi(\conesupp V)\,.
$$
\end{cor}

Theorem \ref{t:composition} and Corollary \ref{c:adjoint} immediately imply

\begin{cor}\label{c:composition}
Let $V_1$, $V_2$ be as in Theorem {\rm\ref{t:composition}}. Assume that either 
$V_1,V_2\in\AC'_0$ or $V_2\in\AC_0$. Then $V_2V_1$ is a FIO of order $m_1+m_2$ associated with the transformation $\Phi=\Phi_2\circ\Phi_1$ with principal symbol equal to 
$\,\sigma_{{V_2} {V_1}} (y,\eta,\mathfrak{b}) = 
\sigma_{V_1}(y,\eta,\mathfrak{a}_1)\,\sigma_{V_2}(\Phi_1(y,\eta),\mathfrak{a}_2)\,$ such that
$$
\conesupp(V_2V_1)\subset\left(\conesupp V_1\,\bigcap\,\Phi_1^{-1}(\conesupp V_2)\right).
$$
\end{cor}

\begin{rem}\label{r:simple}
In simple words, the above formulae for principle symbols mean that $\,\sigma_{{V_2}^*{V_1}} (y,\eta) = 
i^{k_1}\,\sigma_{V_1}(y,\eta)\,\overline{\sigma_{V_2}(\Phi(y,\eta))}\,$, $\,\sigma_{V^*}(y,\eta)=i^{k_2}\,\overline{\sigma_V(\Phi^{-1}(y,\eta))}\,$ and 
$\,\sigma_{{V_2} {V_1}} (y,\eta) = i^{k_3}\,
\sigma_{V_1}(y,\eta)\,\sigma_{V_2}(\Phi_1(y,\eta))\,$, where $k_j$ are integers which are uniquely determined by the canonical transformations.
\end{rem}

It is well known that FIOs can be extended to the space $\mathcal{E}'(M\setminus\partial M)$ of distributions with compact supports in 
$M\setminus\partial M$ (see, for instance, \cite{Sh} or \cite{Tr}). Theorem \ref{t:composition} and standard results on $\PsDO$s imply that a FIO of order 
$m$ lying in $\AC_0$ is bounded from $H^s(M)$ to $H^{s-m}(M)$.

Let $A,B\in\AC_0$ be $\PsDO$s, and let $V$ be a FIO associated with a canonical transformation $\Phi$. Then, by Corollary \ref{c:composition}, $AVB$ is a 
FIO associated with $\Phi$ such that $\conesupp(AVB)$ is a subset of
\begin{equation*}
\conesupp B\bigcap\conesupp V\,\bigcap\,\Phi^{-1}(\conesupp A).
\end{equation*}
The above inclusion implies that ${\mathrm WF}\,(Vu)\subset \Phi\left({\mathrm WF}\,(u)\right)$ for all distributions $u\in\mathcal{E}'(M\setminus\partial M)$, 
where ${\mathrm WF}\,(\cdot)$ denotes the wave front set. Roughly speaking, this means that singularities of a distribution are moved by the map $\Phi$ under the action of the associated FIO $V$.

The following is a refined version of Egorov's theorem.

\begin{theorem}\label{t:egorov}
Let $V_1$ and $V_2$ be FIOs of orders $m_1$ and $m_2$ associated with a canonical transformation $\Phi$. If $A\in\AC_0$ is a $\PsDO$ of order $m$ then
\begin{enumerate}
\item[(1)]
the composition $B:=V_2^*AV_1$ is a $\PsDO$ of order $m+m_1+m_2$ such that
$$
\conesupp B\subset\conesupp V_1\,\bigcap\,\Phi^{-1}(\conesupp A)\,\bigcap\,\conesupp V_2;
$$
\item[(2)]
$\,\sigma_B(y,\eta,\mathfrak{a})=\sigma_{V_1}(y,\eta,\mathfrak{a})\,
\sigma_A(\Phi(y,\eta))\,\overline{\sigma_{V_2}(y,\eta,\mathfrak{b})}\,$
\item[(3)]
if $\,V_1=V_2\,$ then $\sigma_B(y,\eta)=\sigma_A(\Phi(y,\eta))\,|\sigma_{V_1}(y,\eta)|^2$.
\end{enumerate}
\end{theorem}

\begin{proof}
The first two statements are obtained by applying Corollary \ref{c:composition} to the composition $AV_1$ and then Theorem \ref{t:composition} to the composition of this operator with $V_2^*$.

If $V_1=V_2$ then $\sigma_B(y,\eta)$ coincides with $\sigma_A(\Phi(y,\eta))\,|\sigma_{V_1}(y,\eta)|^2$ as the principal principal symbol of a FIO, that is, modulo a factor 
$i^k$ with $k\in\zbb_4$. If $A$ is a non-negative self-adjoint operator then so is $B=V_1^*AV_1$. It follows that $\sigma_B\geq0$, which implies that $k=0$. Since $\sigma_B$ 
continuously depends on $A$, the same is true for all $\PsDO$s $A$.
\end{proof}

\section{The unitary group $e^{-\I t \Delta^{1/2}}$}\label{s:group}

In this section and further on we shall denote
$$
U(t)\ :=\ e^{-\I t \Delta^{1/2}}\,.
$$
Since $\Delta$  is self-adjoint, the operators $U(t)$ form a strongly continuous
unitary group in the space of half-densities in $L^2(M,\Omega^{1/2})$.

\subsection{Representation by FIOs}
In the following theorem $\Psi^{-\infty}$ denotes the class of operators whose Schwartz kernels are infinitely smooth on $[0,T]\times M \times M$ up 
to the boundary.

\begin{theorem}\label{t:parametrix}
Let $A$ and $B$ be $\PsDO$s in the space of half-densities on $M$  such that $A\in\AC''_0$, $B\in\AC'_0$ and $\conesupp B\subset\OC_{T}$.
Then, on the time interval $[0,T]$, we have $AU(t)B=\sum_{\kappa,j}AU_{\kappa,j}(t)B$ modulo $\Psi^{-\infty}$, where the sum is finite and $U_{\kappa,j}(t)$ are parameter-dependent FIOs such that
\begin{enumerate}
\item[(1)]
each FIO $U_{\kappa,j}(t)$ is associated with the canonical transformation
$\Phi_{\kappa,j}^t:(y,\eta)\mapsto(x^t_{\kappa,j}(y,\eta),\xi^t_{\kappa,j}(y,\eta))$, where $(y,\eta)\in\conesupp B$ and
$(x^t_{\kappa,j},\xi^t_{\kappa,j})$ is a billiard segment defined in Subsection {\rm\ref{s:billiards2}};
\item[(2)]
each FIOs $U_{\kappa,j}(t)$ satisfies the equation $\partial_t^2U_{\kappa,j}(t)-\Delta U_{\kappa,j}(t)=0$ modulo $\Psi^{-\infty}$;
\item[(3)]
the FIO $U_{\kappa,0}(t)$ associated with the first segment (which starts at $t=0$) satisfies the initial condition $U_{\kappa,0}(0)=I$, for all other segments $U_{\kappa,j}(0)=0$;
\item[(4)]
the phase functions corresponding to the incoming, reflected and refracted trajectories coincide on the set $\left\{x\in Y\bigcup\partial X\right\}$, and so do the arguments $\arg(\det^2\varphi_{x\eta})$;
\item[(5)] the principal symbols $\sigma_{U_{\kappa,j}(t)}$ are (locally) independent of $t$. 
\end{enumerate}
\end{theorem}

The theorem is proved using the standard technique, which goes back to \cite{Ch}. For a usual elliptic boundary value problem with branching billiards, it is discussed in detail in \cite{SV2}. 
For the operator defined in Section \ref{s:setting}, a sketch of proof was given in \cite{Sa1}.

\begin{rem}\label{r:negative-time}
The same arguments show that a similar result holds for negative times. More precisely, Theorem {\rm\ref{t:parametrix}} remains valid for $t\in[-T,0]$ under the assumption that $\conesupp B\subset\OC_{m,T}^-$, where $\OC_T^-$ is the set of starting points of billiard trajectories going in the reverse direction which are well defined on the time interval $[-T,0]$. One can easily show that $(y,\eta)\in\OC_T^-$ if and only if $(y,-\eta)\in\OC_T$.
\end{rem}

\subsection{Principal symbols of the FIOs $U_{\kappa,j}(t)$}

Following \cite[Section 2.6.3]{SV2}, let us fix branches of $\arg(\det^2\varphi_{x\eta})$ for the FIOs $U_{\kappa,j}(t)$ assuming that
\begin{itemize}
\item 
for the FIO $U_{\kappa,0}(t)$ associated with the first segment of billiard trajectories, $\left.\arg(\det^2\varphi_{x\eta})\right|_{t=0,x=y}=0$;
\item 
the branches corresponding to the incoming, reflected and refracted trajectories coincide on the set $\{x\in Y\bigcup\partial X\}$.
\end{itemize}
In view of the parts (3) and (4) of Theorem \ref{t:parametrix}, these two conditions can be satisfied. Clearly, they uniquely determine the branch of $\arg(\det^2\varphi_{x\eta})$ for all the FIOs $U_{\kappa,j}(t)$ for all times $t$. It follows that the bundles 
$\DC_{\mathbb{Z}}(\Phi_{\kappa,j}^t)$ over $\conesupp B$ associated with the transformations $\Phi_{\kappa,j}^t$ are trivial.
This allows us to consider the principal symbols $\sigma_{\kappa,j}(t;y,\eta)$ of FIOs $U_{\kappa,j}(t)$ as single-valued functions on $\rbb_+\times T^*M$. In particular, Theorem \ref{t:parametrix}(5) implies that the principal symbol  $\sigma_{U_{\kappa,0}(t)}$ associated with the first segment is identically equal to 1.

Let $\OC_{m,T}$ be the conic set of points $(y,\eta)\in \OC_{T}$ such that all the billiard trajectories $(x^t_\kappa(y,\eta),\xi^t_\kappa(y,\eta))$ experience at most $m$ reflections and refractions for $t\in[0,T]$.
The sets $\OC_{m,T}$ are open, and their union over $m$ contains $\OC_{T}$.
Since the intersection $S^*M\bigcap\conesupp B$ is compact, the conic support $\conesupp B$ is covered by a finite collection of connected components of  $\OC_{m,T}$ with  a sufficiently large $m$. Assume, for the sake of simplicity, that $\conesupp B$ lies in one connected component of $\OC_{m,T}$, and let
$U_{\kappa,j}(t)$ be one of the FIOs introduced in Theorem \ref{t:parametrix}. Since there are no dead-end, grazing or singular trajectories of length $t\in[0,T]$ originating from $\OC_{m,T}$, for each fixed $\kappa$ and $j$ there are three possibilities :
\begin{enumerate}
\item[(i)] 
all segments $(x^t_{\kappa,j}(y,\eta),\xi^t_{\kappa,j}(y,\eta))$ with $(y,\eta)\in\conesupp B$ end at the points of $\left.T^*M\right|_{Y_{\mathrm{reg}}}$ which do not belong to the set of total reflection; 
\item[(ii)]
all segments $(x^t_{\kappa,j}(y,\eta),\xi^t_{\kappa,j}(y,\eta))$ with $(y,\eta)\in\conesupp B$ end at the points of total reflection in $\left.T^*M\right|_{Y_{\mathrm{reg}}}$;
\item[(iii)]
all segments $(x^t_{\kappa,j}(y,\eta),\xi^t_{\kappa,j}(y,\eta))$ with $(y,\eta)\in\conesupp B$ end at $\left.T^*M\right|_{\partial X_{\mathrm{reg}}}$.
\end{enumerate}

In the first case, in order to satisfy the boundary condition, one has to add to $U_{\kappa,j}(t)$ two FIOs $U_{\kappa_0,j+1}(t)$ and $U_{\kappa_1,j+1}(t)$, corresponding to the reflected and refracted trajectories, respectively. 

In the second case, the FIO $U_{\kappa_0,j+1}(t)$ corresponding to the reflected trajectories is chosen is such a way that the Schwartz kernel of the sum $U_{\kappa,j}(t)+U_{\kappa_0,j+1}(t)$ has singularities only at the points $(t,x,y)$ with $x\in Y$. After that one can satisfy the boundary condition by adding a `boundary layer term' whose singularities are also located at the points $(t,x,y)$ with $x\in Y$ (see \cite[Section 3.3.4]{SV2} for details). Since $A\in\AC''_0$, the boundary layer terms do not appear in the sum representing the operator $AU(t)B$.

Finally, in the third case, the Dirichlet or Neumann boundary condition is satisfied by adding only a FIO $U_{\kappa_0,j+1}(t)$ corresponding to the reflected trajectories.

Let $\sigma_{\kappa,j}(t;y,\eta)$, $\sigma_{\kappa_0,j+1}(t;y,\eta)$ and $\sigma_{\kappa_1,j+1}(t;y,\eta)$ be the principal symbols of the FIOs $U_{\kappa,j}(t)$, $U_{\kappa_0,j+1}(t)$ and $U_{\kappa_1,j+1}(t)$.

\begin{lem}\label{l:split}
Assume that the trajectory $(x^t_{\kappa,j}(y,\eta),\xi^t_{\kappa,j}(y,\eta)$ approaches $Y_{\mathrm{reg}}$ from the $g_+$-side and hits $Y_{\mathrm{reg}}$ at the time $t^*(y,\eta)$ at the point 
$$
(x^*,\xi^*)\ :=\ \lim_{t\to t^*(y,\eta)-0}(x^t_{\kappa,j}(y,\eta),\xi^t_{\kappa,j}(y,\eta)\,.
$$
Let $\tau_\pm$ be as in Subsection {\rm\ref{s:billiards1},} and let $\tilde\tau_-=\sqrt{|g_+(x^*,\xi^*)-g_-(x^*,\xi_Y^\pm)|}$ if $(x^*,\xi^*)$ is a point of total reflection. Define
$b_+:=b_+(x^*).$
Then 
\begin{eqnarray}
\sigma_{\kappa_0,j+1}(t^*;y,\eta)&=&
\tau_{\kappa_0,j+1}(y,\eta)\,\sigma_{\kappa,j}(t^*;y,\eta)\,,\label{coef:reflection}
\\
\sigma_{\kappa_1,j+1}(t^*;y,\eta)&=&
\tau_{\kappa_1,j+1}(y,\eta)\,\sigma_{\kappa,j}(t^*;y,\eta)\,,\label{coef:refraction}
\end{eqnarray}
where
\begin{enumerate}
\item[(i)]
in the first case $\,\tau_{\kappa_0,j+1}=\frac{b_+  \tau_+-\tau_-}{b_+ \tau_++\tau_-}\,$ and 
$\,\tau_{\kappa_1,j+1}\ =\  \frac{2\,\sqrt{b_+ \tau_+\,\tau_-}}{b_+ \tau_++\tau_-}\,$;
\item[(ii)]
in the second case $\,\tau_{\kappa_0,j+1}=\frac{b_+ \tau_+-i\,\tilde\tau_-}{b_+ \tau_++i\,\tilde\tau_-}\,$.
\end{enumerate}
If the trajectory hits the boundary $\partial X_{\mathrm{reg}}$ then we have \eqref{coef:reflection} with
\begin{enumerate}
\item[(iii)]
$\tau_{\kappa_0,j+1}=1$ for the Neumann boundary condition, and\\ $\tau_{\kappa_0,j+1}=-1$ for the Dirichlet boundary condition.
\end{enumerate}
\end{lem}
 
Clearly, if the geodesics approaches $Y_{\mathrm{reg}}$ from the $g_-$-side, then the coefficients $\,\tau_{\kappa_0,j+1}\,$ and $\,\tau_{\kappa_1,j+1}\,$ are defined by the formulae which are obtained from the above equalities by swapping $+$ and $-\,$.

\begin{rem}\label{r:energy}
Note that $\left(\tau_{\kappa_0,j+1}\right)^2+\left(\tau_{\kappa_1,j+1}\right)^2=1$ in the case {\rm(i),} and $\left|\tau_{\kappa_0,j+1}\right|=1$ in the cases {\rm(ii)}, {\rm(iii)}.
\end{rem}

Part {\rm(iii)} of Lemma \ref{l:split} is a particular case of \cite[Corollary 3.4.7]{SV2}. The formulae {\rm(i)} and {\rm(ii)} can be deduced from 
\cite[Proposition 3.3]{Sa1}, which states the same result but for principals symbols defined in a different way. However, the proof of Proposition 3.3 in \cite{Sa1} is very sketchy and is not easy to reconstruct. Therefore in Appendix \ref{s:b} we outline a direct proof, which uses the technique developed in \cite{SV2}.

\subsection{The index function of billiard transformations}\label{s:thetas}

In this subsection we shall briefly recall some results from \cite[Section 1.5]{SV2} and \cite[Appendix D.6]{SV2}. Strictly speaking, they were proved in \cite{SV2} only for billiards obtained by reflections. However, the same technique of matching phase functions is applicable to refracted trajectories, and the proofs remain exactly the same.

Let us denote the index functions of the transformations $\Phi^t_{\kappa,j}$ by $\Theta_{\kappa,j}(t;y,\eta)$. Theorem \ref{t:parametrix}(4) implies that the index functions corresponding to two consecutive billiard segments coincide at the points of reflection and refraction. This allows us to define the index function $\Theta_\kappa(t;y,\eta)$ associated with the transformation  $\Phi^t_\kappa:(y,\eta)\mapsto(x^t_\kappa,\xi^t_\kappa)$.

Consider the matrix of the first derivatives  $(x^t_\kappa)_\eta(y,\eta)$. Since $x^t_\kappa$ is positively homogeneous in $\eta$ of degree zero, $\rank(x^t_\kappa)_\eta$ is not greater than $n-1$. If $\rank(x^s_\kappa)_\eta<n-1$ for some $s>0$ then the point  $x^s_\kappa$ is said to be a conjugate point of the billiard trajectory $x^t_\kappa(y,\eta)$, and the number $n-1-\rank(x^s_\kappa)_\eta$ is called its multiplicity. At the points of reflection and refraction, $\rank(x^t_\kappa)_\eta$ is the same for the incoming, reflected and refracted trajectories. Thus the notions of a conjugate point and its multiplicity are well defined  for all $(y,\eta)\in\OC_T$ and all $s\in[0,T]$.

The following statement is an immediate corollary of \cite[Lemma 1.5.6]{SV2} and  \cite[Theorem D.6.8]{SV2}.

\begin{pro}\label{p:theta}
If $(y,\eta)\in\OC_T$ then every billiard trajectory $x^t_\kappa(y,\eta)$, $t\in[0,T]$, has only finitely many conjugate points. The number of its conjugate points counted with their multiplicities is equal to $-\Theta_\kappa(T;y,\eta)$.
\end{pro}

\section{Local Weyl asymptotics}

The operator $\Delta$ is self-adjoint and has compact resolvent and therefore there exists an orthonormal basis
$(\phi_j)_{j \in \mathbb{N}}$ in $L^2(M,\Omega^{1/2})$ such that
$$
 \Delta \phi_j = \lambda_j \phi_j,
$$
where $0 \leq \lambda_1 \leq \lambda_2 \leq \ldots \to \infty$.

Given an $L^2$-bounded operator $A$, let us define 
$$
N_A(\lambda):=\sum_{\lambda_j<\lambda^2}\langle A\phi_j,\phi_j\rangle.
$$ 
The Weyl asymptotic formula for $N(\lambda)$ implies that 
$$
N_A(\lambda)\ \leq\ \|A\|\,N(\lambda)\ \leq\ \mathrm{const}\,\|A\|\,\lambda^n\,.
$$
We shall denote
$$
\Lambda(A)\ :=\ \lim_{N\to\infty}\frac{1}{N}\sum_{j=1}^N\langle A \phi_j,\phi_j \rangle\ =\ 
\lim_{\lambda\to\infty}\frac1{N(\lambda)}\sum_{\lambda_j<\lambda^2}\langle A \phi_j,\phi_j \rangle
$$
whenever the limit exists. Formulae giving the value of $\Lambda(A)$ in terms of other characteristics 
of the operator $A$ (such as its symbol or Schwartz kernel) are usually called local Weyl laws.

\begin{rem}\label{Calkin}
Since $\Lambda(K)=0$ whenever the operator $K$ is compact, $\Lambda(A)$ depends only on the image of $A$ in the Calkin algebra. 
\end{rem}

Clearly, $N_A(\lambda)$ is a function with locally bounded variation, whose derivative $N'_A$ is the sum of 
$\delta$-functions located at the points $\lambda_j^{1/2}$ with coefficients $\langle A\phi_j,\phi_j\rangle$.
For every smooth rapidly decreasing function $\rho$ on $\rbb$ we have 
\begin{equation}\label{rho1}
\sum_{j=1}^\infty\rho(\lambda-\lambda_j^{1/2})\,\langle A\phi_j,\phi_j\rangle\ =\ \rho*N'_A(\lambda)\ =\ \rho'*N_A(\lambda)\,,
\end{equation}
where $*$ denotes the convolution. In many cases it is easier to investigate the asymptotic behaviour of $\rho*N'_A(\lambda)$. After that, a local Weyl law can be obtained by applying a suitable Tauberian theorem.

If $A\in\AC_0$ and its Schwartz kernel $\AC(x,y)$ satisfies the condition
\begin{enumerate}
\item[{\bf(C)}]
$(x,\xi;y,0)\not\in\WF(\AC)$ for all $(x,\xi)\in T^*M$ and $y\in M$
\end{enumerate}
then the trace $\Tr\left(AU(t)\right)$ exists as a distribution in $t$. Indeed,
the above condition implies that $AU(t)(1+\Delta)^{-m}$ is an operator with a continuous kernel $\UC_{A,m}(t,x,y)$ for all sufficiently large positive integers $m$, so that we can define 
$$
\Tr\left(AU(t)\right)\ :=\ \left(1-\,\frac{\dr^2}{\dr t^2} \right)^m \int_M \UC_{A,m}(t,y,y)\,\dr y\,.
$$ 
In this case
\begin{multline}\label{rho2}
\rho*N'_A(\lambda)
=\FC^{-1}_{t\to\lambda}\left(\sum_{j=1}^\infty\hat\rho(t)\,e^{-it\lambda_j}\,\langle A\phi_i,\phi_j\rangle\right)\\
=\FC^{-1}_{t\to\lambda}\left(\hat\rho(t)\,\Tr \left(AU(t)\right)\right)=\FC^{-1}_{t\to\lambda}\left(\hat\rho(t)\,\Tr \left(U(t)A\right)\right), 
\end{multline}
where $\hat\rho$ is the Fourier transform of $\rho$ and $\FC^{-1}_{t\to\lambda}$ is the inverse Fourier transform.

Let $V\in\AC_0$ be a FIO of order zero associated with canonical transformations $\Phi$. By \cite[Theorem 8.1.9]{H2}, the wave front of the Schwartz kernel of $V$ is a subset of
$$
\{(x,\xi;y,-\eta)\in T^*M \times T^*M \,:\, (x,\xi)=\Phi(y,\eta)\}\,,
$$
Therefore $V$ satisfies the condition {\bf(C)}.

Further on
\begin{itemize}
\item
$\displaystyle{\dr\omega(y,\tilde\eta)\ =\ \frac{\dr y\,\dr\tilde\eta}{\mathrm{Vol}(S^*M)}}\ $ is the normalised measure on $S^*M$;
\item
$\mathrm{Fix}(\Phi)=\{(y,\tilde\eta)\in S^*M \,:\, \Phi(y,\tilde\eta)=(y,\tilde\eta)\}$ is the set of fixed points of a transformation $\Phi$ lying in $S^*M$.
\end{itemize}

The following two results are proved in Appendix \ref{s:c}.

\begin{theorem}\label{t:local-fio}
If $\,V\in\AC_0$ is a FIO of order zero associated with a homogeneous canonical transformations $\Phi$ then the limit $\Lambda(V)$ exists and
\begin{equation}\label{Lambda}
\Lambda(V)\ =\ \int_{\mathrm{Fix}(\Phi)}i^{\Theta_\Phi(y,\tilde\eta)}\sigma_V(y,\tilde\eta)\,\dr\omega(y,\tilde\eta)\,.
\end{equation}
\end{theorem}

\begin{lem}\label{l:thetas}
Let $\,V_j\in\AC_0$ be FIOs of order zero associated with homogeneous canonical transformations $\Phi_j$. If $V=V_2^* V_1$ and $\Phi=\Phi_2^{-1}\circ\Phi_1$ then 
$$i^{\Theta_\Phi} \sigma_V=( i^{\Theta_{\Phi_1}} \sigma_{V_1}) ( i^{-\Theta_{\Phi_2}}\overline{\sigma_{V_2}})\,$$ 
almost everywhere on the set $\mathrm{Fix}(\Phi)$. 
\end{lem}

Theorems \ref{t:composition}, \ref{t:local-fio}, and Lemma \ref{l:thetas} immediately imply 

\begin{cor}\label{c:local-fios}
If $\,V_j\in\AC_0$ are FIOs of order zero associated with canonical transformations $\Phi_j$ 
then the limit $\Lambda(V_2^*V_1)$ exists and coincides with
$$
\Lambda(V_2^*V_1)\ =\ \int_{\mathrm{Fix}(\Phi)}i^{\Theta_{\Phi_1}(y,\tilde\eta)}\,\sigma_{V_1}(y,\tilde\eta)\,i^{-\Theta_{\Phi_2}(y,\tilde\eta)}\,\overline{\sigma_{V_2}(y,\tilde\eta)}\,\dr\omega(y,\tilde\eta)\,,
$$
where $\Phi=\Phi_2^{-1}\circ\Phi_1$.
\end{cor}

Applying Theorem \ref{t:local-fio} with $\Phi=I$ and the Weyl formula for the counting function, we obtain the following well known result.

\begin{cor}\label{c:local-pdo}
If $A\in\AC$ is a $\PsDO$ of order zero then the limit $\Lambda(A)$ exists and is equal to
$\int_{S^*M}\sigma_A(y,\tilde\eta)\,\dr\omega(y,\tilde\eta)$.
\end{cor}

\begin{rem}
A theorem similar to Theorem {\rm\ref{t:local-fio}} was stated and proved under a clean intersection condition in \cite{Zelditch:92}. 
Note that the formula for $\Lambda(V)$ given in \cite{Zelditch:92} differs from \eqref{Lambda}, as it does not contain the factor $i^{\Theta_\Phi}$.
This is due to the fact that the author used an implicitly defined notion of "scalar principal symbol". An explicit definition of this object would have
involved the index function or an analogue. 
\end{rem}

\section{Classical dynamics of branching billiards}

\subsection{Definitions} \label{ssec:defi}
Let $\OC_T$ be the conic subsets of the cotangent bundle $T^*(X\setminus\partial M)$ defined at Subsection \ref{s:billiards3}. By Assumption \ref{a:bad-points}, $\OC_T$ is an open set  of full measure for each $T\geq0$.

If $(y,\eta)\in\OC_T$ then all the billiard trajectories $(x^t_\kappa(y,\eta),\xi^t_\kappa(y,\eta))$ originating from $(y,\eta)$ are well defined for $t\in[0,T]$ and experience only finitely many reflections and refractions. It follows that, for each fixed $(y,\eta)\in\OC_T$ and $t\in[0,T]$, the set of end points of the trajectories $(x^t_\kappa(y,\eta),\xi^t_\kappa(y,\eta))$ is finite. Let us denote it by $\Phi^t(y,\eta)$.

\begin{rem}\label{r:dynamics1}
Note that $\,\xi^{t^*}_\kappa(y,\eta)$ is not uniquely defined if the trajectory hits the boundary at the time $t^*$. For the sake of definiteness, we shall be assuming that in this situation $\xi^{t^*}_\kappa(y,\eta):=\lim_{t\to t^*-0}\xi^t_\kappa(y,\eta)$.
\end{rem}

\begin{rem}\label{r:dynamics2}
Recall that the shifts along billiard trajectories $\Phi^t_\kappa:(y,\eta)\mapsto(x^t_\kappa,\xi^t_\kappa)$ are homogeneous canonical transformations in $T^*X$.
One can consider the branching billiard system as a family of multi-valued canonical transformations $\Phi^t$, mapping $(y,\eta)\in T^*X$ into the set $\Phi^t(y,\eta)=\bigcup_\kappa\Phi^t_\kappa(y,\eta)$. 
\end{rem}

Suppose that a billiard trajectory $(x^t_\kappa(y,\eta),\xi^t_\kappa(y,\eta))$ is well defined and hits the boundary at the times $0<t^*_1(y,\eta)<t^*_2(y,\eta)\ldots$ For each $t^*_j(y,\eta)$ we have the associated coefficient
$\tau_{\kappa,j+1}(y,\eta)$ calculated in Lemma \ref{l:split}, where $j+1$ is the order number of the next segment. More precisely, in the notation of Lemma \ref{l:split}, $\tau_{\kappa,j+1}:=\tau_{\kappa_0,j+1}$ if the next segment is obtained by reflection, and $\tau_{\kappa,j+1}:=\tau_{\kappa_1,j+1}$ if the next segment is obtained 
by refraction. 

Let us define 
$$
\tau_{\kappa}(t;y,\eta)\ :=\ 
\begin{cases}
1 & \text{if}\ 0\leq t\leq t^*_1(y,\eta)\,,\\
\prod_{t^*_j(y,\eta)<t}\, \tau_{\kappa,j+1}(y,\eta)& \text{if}\ t^*_1(y,\eta)<t\,.
\end{cases}
$$
In view of  Remark \ref{r:energy},
\begin{equation}\label{tau2}
\sum_{(x^t_\kappa,\xi^t_\kappa)} |\tau_{\kappa}(t;y,\eta) |^2\ =\ 1
\end{equation}
for all $(y,\eta)\in\OC_T$ and $t\in[0,T]$, where the sum is taken over all distinct billiard trajectories of `length' $t$ originating from $(y,\eta)$. 

\begin{rem}\label{r:weight}
We call the number $| \tau_{\kappa}(t;y,\eta) |^2$ the weight of the trajectory $(x^s_\kappa(y,\eta),\xi^s_\kappa(y,\eta))$, $s\in[0,t]$. It can be thought of as the proportion of energy transmitted along the billiard trajectory, or the probability for a particle to travel along this trajectory.
\end{rem}

If $(x,\xi)\in\Phi^t(y,\eta)$, let us denote
\begin{align*}
w^{\mathrm c}_{(x,\xi)}(t;y,\eta)\ &:=\ \sum_{(x^t_\kappa,\xi^t_\kappa)=(x,\xi)}|\tau_{\kappa}(t;y,\eta)|^2,\\
w^{\mathrm d}_{(x,\xi)}(t;y,\eta)\ &:=\ \bigl|\sum_{(x^t_\kappa,\xi^t_\kappa)=(x,\xi)}i^{\Theta_\kappa(t;y,\eta)}\, \tau_{\kappa}(t;y,\eta)\;\bigr|^2,
\end{align*}
where $\tau_{\kappa}(t;y,\eta)$ are as above, $\Theta_\kappa$ are the index functions introduced in Subsection \ref{s:thetas},
and the sum is taken over all distinct billiard trajectories of `length' $t$ originating from $(y,\eta)$ and ending at $(x,\xi)$.  

In view of Assumption \ref{a:bad-points}, the following definition makes sense for all $p\in[1,\infty]$ and $t\geq0$.

\begin{definition} \label{def:transfer-op}
The classical transfer operators $\Xi_t^{\mathrm c}$ and the diagonal transfer operators $\Xi_t^{\mathrm d}$ in the space $L^p(S^*M,\dr\omega)$ are defined for times $t\geq0$ by the equalities
\begin{align*}
(\Xi_t^{\mathrm c} f)(y,\tilde\eta)\ &:= \sum_{(x,\xi)\in\Phi^t(y,\tilde\eta)}w^{\mathrm c}_{(x,\xi)}(t;y,\tilde\eta)\,f(x,\xi)\,,\\
(\Xi_t^{\mathrm d} f)(y,\tilde\eta)\ &:= \sum_{(x,\xi)\in\Phi^t(y,\tilde\eta)}w^{\mathrm d}_{(x,\xi)}(t;y,\tilde\eta)\,f(x,\xi)\,,
\end{align*}
where $(y,\tilde\eta)\in S^*(M\setminus\partial M)\setminus\OC_t\,$.
\end{definition}

The difference between $\Xi^{\mathrm d}$ and $\Xi^{\mathrm c}$ is in the contributions from recombining billiard trajectories, that is, the billiard trajectories such that $(x^t_{\kappa},\xi^t_{\kappa})=(x^t_{\kappa'},\xi^t_{\kappa'})$ but $(x^s_{\kappa},\xi^s_{\kappa})\ne(x^s_{\kappa'},\xi^s_{\kappa'})$ for some $s\in(0,t)$.
If the set of initial points that admit recombining billiard trajectories has measure zero then  $\Xi_t^{\mathrm d}$ and $\Xi_y^{\mathrm c}$ coincide. 

\begin{exa}\label{e:hemispheres}
Let $\tilde X$ be the unit 2-dimensional sphere and $Y$ be a great circle, splitting $\tilde X$ into the union of two hemispheres $X_+$ and $X_-$. Let us provide $X_\pm$ with the metrics $c_\pm\,g$, where $g$ is the standard metric on $\tilde X$ and $c_\pm$ are positive constants, and consider the Riemannian manifold $X=X_+\bigcup X_-$. In this situation, the billiard trajectories are formed by the great semicircles lying either in $X_+$ or $X_-$, whose length is equal to $\pi\,c_+$ and $\pi\,c_-$, respectively. If $m_+c_+=m_-c_-$ with some positive integers $m_\pm$ and $t>4\pi m_+c_+$ then for every billiard trajectory $(x^t_{\kappa}(y,\tilde\eta),\xi^t_{\kappa}(y,\tilde\eta))$ there exists a distinct trajectory $(x^t_{\kappa'}(y,\tilde\eta),\xi^t_{\kappa'}(y,\tilde\eta))$ with the same end point, which is obtained from $(x^t_{\kappa}(y,\tilde\eta),\xi^t_{\kappa}(y,\tilde\eta))$ by replacing $2m_+$ great semicircles in $X_+$ with $2m_-$ great semicircles in $X_-$ or the other way round. If $c_+$ and $c_-$ 
are rationally independent then there are no recombining trajectories.
\end{exa}

\subsection{Results}
The following theorem reveals the link between the diagonal transfer operators and local Weyl asymptotics.

\begin{theorem} \label{averagingtheorem}
Let $B,C\in\AC_0$ be $\PsDO$s of order zero in the space of half-densities on $M$ 
such that $\conesupp B\bigcup\conesupp C\subset\OC_T$.
Then, for all $\PsDO$s $A\in\AC$ of order zero and all $t \in [0,T]$, 
\begin{equation*}
 \Lambda\left(C U^*(t) A U(t) B\right) \ =\
  \int_{S^*M} \sigma_C(y,\tilde\eta)\;\Xi^{\mathrm d}_t(\sigma_A)(y,\tilde\eta)\;\sigma_B(y,\tilde\eta)\;\dr\omega(y,\tilde\eta)\,.
\end{equation*}
\end{theorem}

\begin{proof}
If $A$ is the multiplication by a constant then the theorem follows from Corollary \ref{c:local-pdo}. Thus we can assume without loss of generality that $A\in\AC_0$.

Then, by Theorem \ref{t:parametrix}, $U(t)B$ and $U(t)C^*$ can be represented as finite sums of FIOs $\sum_{\kappa,j}U_{\kappa,j}(t)B$ and $\sum_{\kappa',j'}U_{\kappa',j'}(t)C^*$. It follows that
\begin{equation}\label{sum}
C U^*(t) A U(t) B\ =\ \sum_{\kappa,\kappa',j,j'}CU^*_{\kappa',j'}(t)AU_{\kappa,j}(t)B\,,
\end{equation}
Corollaries \ref{c:adjoint} and \ref{c:composition} imply that $(CU^*_{\kappa',j'})^*=U_{\kappa',j'}C^*$ and $AU_{\kappa,j}(t)B$ are FIOs associated with canonical transformation $\Phi^t_{\kappa',j'}$ and $\Phi^t_{\kappa,j}$ with principal symbols 
$$
\sigma_{U_{\kappa',j'}(t)}(y,\eta)\,\overline{\sigma_C(y,\eta)}\quad\text{and}\quad
\sigma_A(\Phi^t_{\kappa,j}(y,\eta))\,\sigma_{U_{\kappa,j}(t)}(y,\eta)\,\sigma_B(y,\eta),
$$
respectively. Clearly, the set of fixed points of the mapping $(\Phi^t_{\kappa',j'})^{-1}\Phi^t_{\kappa,j}$ consist of $(y,\eta)\in T^*M$ such that
$$
(x^t_{\kappa,j}(y,\eta),\xi^t_{\kappa,j}(y,\eta))=(x^t_{\kappa',j'}(y,\eta),\xi^t_{\kappa',j'}(y,\eta)).
$$
Now the required result is obtained by calculating the principal symbols of $U_{\kappa,j}(t)$ and $U_{\kappa',j'}(t)$ with the use of Lemma \ref{l:split},  applying Corollary \ref{c:local-fios} and summing up over all segments of the billiard trajectories.
\end{proof}

\begin{cor}\label{c:sums}
If $(y,\eta)\in\OC_T$ and $0\leq t\leq T$ then 
$$
\sum_{(x,\xi)\in\Phi^t(y,\eta)}w^{\mathrm c}_{(x,\xi)}(t,y,\eta)=1
\quad\text{and}\quad 
\sum_{(x,\xi)\in\Phi^t(y,\eta)}w^{\mathrm d}_{(x,\xi)}(t,y,\eta)=1\,.
$$
\end{cor}

\begin{proof}
The first equality is an immediate consequence of \eqref{tau2}. The second is proved by applying Theorem \ref{averagingtheorem} to $A=I$ and comparing the obtained result with Corollary \ref{c:local-pdo}.
\end{proof}

\subsection{Properties of the transfer operators}
Clearly, 
\begin{itemize}
\item
$\Xi^{\mathrm c}_t$ and $\Xi^{\mathrm d}_t$ are positivity preserving operators. 
\end{itemize}
Corollary \ref{c:sums} implies that 
\begin{itemize}
\item 
the operators $\Xi^{\mathrm c}_t$ and $\Xi^{\mathrm d}_t$ are continuous in all spaces $L^p(S^*M,\dr\omega)$ with $p\in[1,\infty]$ and their operator norms in these spaces are bounded by 1,
\item 
$\Xi^{\mathrm c}_t$ and $\Xi^{\mathrm d}_t$ are isometries in the space $L^1(S^*M,\dr\omega)$.
\end{itemize}
Note that the operators $\Xi^{\mathrm c}_t$ form a semigroup, whereas $\Xi_t^{\mathrm d}\,\Xi_s^{\mathrm d}$ may not coincide with $\Xi_{t+s}^{\mathrm d}$.

\begin{definition}
Let $\Xi_t$ be either $\Xi_t^{\mathrm c}$ or $\Xi_t^{\mathrm d}$. 
We say that $\Xi_t$ is ergodic if for all $f\in L^\infty(S^*M,\dr\omega)$
\begin{equation}\label{ergodicity1}
2T^{-2} \int_0^T \int_0^t \left( \Xi_s f \right)(y,\tilde\eta)\,\dr s\, \dr t\ \to\ \int_{S^*M} f(y,\tilde\eta)\,\dr\omega(y,\tilde\eta)
\end{equation}
as $T\to+\infty$ almost everywhere in $S^*M$.
\end{definition}

In view of Corollary \ref{c:sums} and Lebesgue's dominated convergence theorem, if $\Xi_t$ is ergodic then
\begin{equation}\label{ergodicity2}
2T^{-2} \int_0^T \int_0^t \left( \Xi_s f \right)\dr s\, \dr t\ \to\ \left( \int_{S^*M} f(y,\tilde\eta)\,\dr\omega(y,\tilde\eta) \right)\mathbf{1}
\end{equation}
in all spaces $L^p(S^*M,\dr\omega)$, where $\mathbf{1}$ is the function identically equal to one.

\begin{rem}\label{r:ergodicity1}
Changing the order of integration, one can rewrite the condition \eqref{ergodicity1} in the following equivalent form
$$
2T^{-2} \int_0^T(T-s) \left( \Xi_s f \right)(y,\tilde\eta)\,\dr s\ \to\ \int_{S^*M} f(y,\tilde\eta)\,\dr\omega(y,\tilde\eta)\,.
$$
\end{rem}

\begin{rem}\label{r:ergodicity2}
The traditional definition of ergodicity assumes that
\begin{equation}\label{ergodicity3}
t^{-1} \int_0^t \left( \Xi_s f \right)\dr s\,\ \to\  \int_{S^*M} f(y,\tilde\eta)\,\dr\omega(y,\tilde\eta)
\end{equation}
as $t\to+\infty$ almost everywhere. It is easy to see that \eqref{ergodicity3} implies \eqref{ergodicity1} but, generally speaking, the converse is not true.
However, if the left hand side of \eqref{ergodicity3} does converge to a limit for all $f$ (as in the von Neumann ergodic theorem) then, by \eqref{ergodicity1}, the limit coincides with $\int_{S^*M} f(y,\tilde\eta)\,\dr\omega(y,\tilde\eta)$ and, consequently, the dynamics is ergodic in the classical sense. In our scenario this happens when $Y=\varnothing$ and there are no branching trajectories.
\end{rem}

\section{Classical ergodicity implies quantum ergodicity} \label{s:ergodicity}

The purpose of this section is to prove the following theorem which is the main result of this paper.

\begin{theorem}\label{theo:main}
Suppose  that Assumption {\rm\ref{a:bad-points}} is fulfilled and that the diagonal dynamics $\Xi_t^{\mathrm d}$ is ergodic. 
Then quantum ergodicity (QE) holds, that is, for any $\PsDO$ $A \in \AC$ of order zero we have
 $$
  \lim_{N\to \infty} \frac{1}{N} \sum_{j=1}^N \left| \langle A \phi_j , \phi_j \rangle
  - \int_{S^*M} \sigma_A(y,\tilde\eta)\,\dr\omega(y,\tilde\eta)\right|\ =\ 0\,.
 $$
\end{theorem}

\begin{proof}
The proof proceeds in several steps.

\subsubsection*{Step 1}
We can assume without loss of generality that 
\begin{equation}\label{1-1}
\int_{S^*M} \sigma_A(x,\xi)\,\dr\omega(y,\tilde\eta)\ =\ 0
\end{equation}
simply by subtracting the constant $\int_{S^*M} \sigma_A(y,\tilde\eta)\,\dr\omega(y,\tilde\eta)$ from $A$. Thus it is sufficient to prove that, under the assumption \eqref{1-1}, 
$$
  \limsup_{N\to \infty} \frac{1}{N} \sum_{j=1}^N \left| \langle A \phi_j , \phi_j \rangle\right|\ =\ 0\,.
$$

\subsubsection*{Step 2} If  $Q$ is an $L^2$-bouned operator, let us denote 
$$
\Lambda_N(Q)\ :=\ \frac{1}{N} \sum_{j=1}^N\langle Q\phi_j,\phi_j\rangle\,,
$$
so that $\Lambda(Q)=\lim_{N\to\infty}\Lambda_N(Q)$ whenever the limit exists.
By the Cauchy--Schwarz inequality,
\begin{equation}\label{2-1}
\frac{1}{N} \sum_{j=1}^N|\langle Q\phi_j,\phi_j\rangle|\ \leq\ \frac{1}{N} \sum_{j=1}^N\|Q\phi_j\|\ \leq\ \left(\Lambda_N(Q^*Q)\right)^{1/2}\,.
\end{equation}

Let us define $A_t:=U(-t)AU(t)$ and 
$\tilde A_T:=T^{-1}\,\int_0^TU(-t)AU(t)\dr t\,$. 
Here and further on integrals of operator-valued functions are understood in the weak sense.

Since $U(t)\phi_j=e^{-it\lambda_j^{1/2}}\phi_j$  and $U(-t)=U^*(t)$ for all $t\in\rbb$, we have 
$$
\langle\tilde A_T\phi_j ,\phi_j \rangle\ =\ \langle A_t\phi_j ,\phi_j \rangle\ =\ \langle A\phi_j , \phi_j \rangle
$$
for all positive integers $j$, $t\in\rbb$ and $T>0$. Therefore, by \eqref{2-1},
$$
\limsup_{N\to\infty}\frac{1}{N} \sum_{j=1}^N \left| \langle A \phi_j , \phi_j \rangle\right|
\ \leq\ \left(\limsup_{N\to\infty}\Lambda_N(\tilde A^*_T\tilde A_T)\right)^{1/2}.
$$
Thus it is sufficient to show that 
\begin{equation}\label{2-2}
\limsup_{N\to\infty}\Lambda_N(\tilde A^*_T\tilde A_T)\ =\ 0.
\end{equation}

\begin{rem}
Note that, generally speaking, the operators 
$A_t$ and $\tilde A_T$ are not FIOs and do not belong to $\AC_0$. Therefore we cannot directly apply Theorem {\rm\ref{t:local-fio}} or Corollary {\rm\ref{c:local-fios}} to evaluate the upper limit \eqref{2-2}.
\end{rem}

\subsubsection*{Step 3} 
Clearly, 
\begin{multline}\label{3-1}
\|\tilde A_T\phi_j\|^2\ =\ T^{-2}\int_0^T\int_0^T\langle U(-t)AU(t)\phi_j ,U(-r)AU(r)\phi_j \rangle\,\dr r\,\dr t\\
=\ T^{-2}\int_0^T\int_0^TF_{A,j}(t-r)\,\dr r\,\dr t\,,
\end{multline}
where 
$$
F_{A,j}(s):=e^{-is\,\lambda_j^{1/2}}
\langle U(-s)A\phi_j ,A\phi_j \rangle=\langle A^*U(-s)AU(s)\phi_j ,\phi_j \rangle\,.
$$
Since $F_{A,j}(-s)=\overline{F_{A,j}(s)}$ and $\|\tilde A_T\phi_j\|^2$ is real, the integral on the right hand side of \eqref{3-1} coincides with
$$
2\,T^{-2}\int_0^T\int_0^tF_{A,j}(t-r)\,\dr r\,\dr t= 2\,T^{-2}\int_0^T\int_0^tF_{A,j}(s)\,\dr s\,\dr t\,.
$$
Thus it follows that 
\begin{equation}\label{3-2}
\langle\tilde A_T^*\tilde A_T\phi_j,\phi_j\rangle\ =\ \|\tilde A_T\phi_j\|^2\ =\ 2\,T^{-2}\int_0^T\int_0^t\langle A^*A_s\phi_j ,\phi_j \rangle\,\dr s\,\dr t\,.
\end{equation}

\subsubsection*{Step 4}

Let $\Op(\chi)$ be the operator of multiplication by a real-valued function $\chi\in C_0^\infty(M\setminus\partial M)$ such that $0\leq\chi\leq1$, and let $B\in\AC_0$ be a $\PsDO$  of order zero.
Since $\|A^*A_s\|\leq\|A\|^2$ and $\|\Op(\chi)\|\leq1$, we have
\begin{multline*}
\left|\langle A^*A_s\phi_j ,\phi_j \rangle\ -\ \langle\Op(\chi)A^*A_s\phi_j ,\phi_j \rangle\right|\ =\ \left|\langle\Op(1-\chi)A^*A_s\phi_j ,\phi_j \rangle\right|\\
=\ \left|\langle A^*A_s\phi_j ,\Op(1-\chi)\phi_j \rangle\right|\ \leq\ \|A\|^2\,\|\Op(1-\chi)\phi_j \|_{L^2(M)}
\end{multline*}
where $\Op(1-\chi)$ is the multiplication by $1-\chi$, and 
\begin{multline*}
\left|\langle\Op(\chi)A^*A_s\phi_j ,\phi_j \rangle\
-\ \langle\Op(\chi)A^*A_sB\phi_j ,\phi_j \rangle\right|\\ 
=\ \left|\langle\Op(\chi)A^*A_s(I-B)\phi_j ,\phi_j \rangle\right|\ \leq\ \|A\|^2\,\|(I-B)\phi_j\|_{L^2(M)}\,.
\end{multline*}
These inequalities and \eqref{3-2} imply the estimate
\begin{multline*}
\left|\langle\tilde A^*_T\tilde A_T\phi_j,\phi_j\rangle\ -\ 2\,T^{-2}\int_0^T\int_0^t\langle \Op(\chi)A^*A_sB\phi_j ,\phi_j \rangle\,\dr s\,\dr t\right|\\
\leq \|A\|^2\,\|\Op(1-\chi)\phi_j \|_{L^2(M)}+\|A\|^2\,\|(I-B)\phi_j\|_{L^2(M)}\,.
\end{multline*}
Now, applying the second inequality \eqref{2-1}, we see that
\begin{multline}\label{4-2}
\left|\Lambda_N(\tilde A^*_T\tilde A_T)\ -\ 2\,T^{-2}\int_0^T\int_0^t \Lambda_N\bigl(\Op(\chi)A^*A_sB\bigr)\,\dr s\,\dr t\right|\\
\leq\ \|A\|^2\,\left(\Lambda_N\bigl(\Op(1-\chi)^2\bigr)\right)^{1/2}+\|A\|^2\,\left(\Lambda_N\bigl((I-B)^*(I-B)\bigr)\right)^{1/2}
\end{multline}
for all $N=1,2,\ldots$

\subsubsection*{Step 5}
Let us fix an arbitrary $\eps>0$. In view of \eqref{1-1}, ergodicity of the diagonal dynamics implies that there exists $T>0$ such that 
\begin{equation}\label{5-1}
\left\|2T^{-2}\int_0^T\int_0^t\Xi^{\mathrm d}_s\sigma_A\;\dr s\,\dr t\right\|_{L^1(S^*M,\dr\omega)}\ <\ \eps\,.
\end{equation}
Let us fix such a positive $T$ and choose the nonnegative function $\chi\in C_0^\infty(M\setminus\partial M)$ and the $\PsDO$ $B\in\AC_0$ of order zero such that
\begin{enumerate}
\item[(a)]
$\|(1-\chi)\|^2_{L^2(S^*M,\dr\omega)}<\eps^2$\,,
\item[(b)]
$\conesupp B\subset\OC_{T}$,
$|\sigma_B|\leq1$ and  $\|1-\sigma_B\|^2_{L^2(S^*M,\dr\omega)}<\eps^2$.
\end{enumerate}
Note that (b) can be satisfied because, by Assumption \ref{a:bad-points}, $ \OC_{T}$ has full measure.

In view of Corollary \ref{c:local-pdo}, the limits 
$\Lambda\bigl((I-\Op(\chi))^2\bigr)$ and $\Lambda\bigl((I-B)^*(I-B)\bigr)$ exist and are smaller than $\eps^2$. Therefore 
\begin{enumerate}
\item[$(*)$]
the right hand side of \eqref{4-2} is estimated by $2\eps\|A\|^2$ for all sufficiently large $N$.
\end{enumerate} 

By Theorem \ref{averagingtheorem}, the limit $\Lambda\bigl(\Op(\chi)A^*A_sB\bigr)$ also exists and is equal to 
\begin{equation}\label{5-2}
\int_{S^*M}\chi(y)\,\overline{\sigma_A(y,\tilde\eta)}\,\,\Xi^{\mathrm d}_s(\sigma_A)(y,\tilde\eta)\,\sigma_B(y,\tilde\eta)\, d\omega(y,\tilde\eta)\,.
\end{equation}
The Lebesgue dominated convergence theorem implies that
\begin{multline}\label{5-3}
\lim_{N\to\infty}2\,T^{-2}\int_0^T\int_0^t \Lambda_N\bigl(\Op(\chi)A^*A_sB\bigr)\,\dr s\,\dr t\\
=\ 2\,T^{-2}\int_0^T\int_0^t\Lambda\bigl(\Op(\chi)A^*A_sB\bigr)\,\dr s\,\dr t\,.
\end{multline}
Substituting \eqref{5-2}, integrating over $s$ and $t$, and taking into account \eqref{5-1}  and (a), (b), we see that the absolute value of \eqref{5-3} is smaller than $\,\eps\sup|\sigma_A|$.
Consequently, 
\begin{enumerate}
\item[$(**)$]
the integral in the left hand side of \eqref{4-2} is estimated by\\
$\,\eps\sup|\sigma_A|\,$ for all sufficiently large $N$.
\end{enumerate}

Applying the estimates $(*)$ and $(**)$ to \eqref{4-2}, we obtain 
$$
\limsup_{N\to\infty}\Lambda_N\bigl(\tilde A^*_T\tilde A_T\bigr)\ \leq\ 2\eps\,\|A\|^2+\eps\sup|\sigma_A|\,.
$$
Since $\eps$ can be chosen arbitrarily small, this implies \eqref{2-2}.
\end{proof}

\appendix

\section{Proof of Theorem \ref{t:composition}}\label{s:a}

In this and next sections, if $f=(f_1,\ldots,f_n)$ is a vector-function of $n$-dimensional variable $\theta=(\theta_1,\ldots,\theta_n)$, we  denote by $f_\theta$ the $n\times n$-matrix function with entries $(f_i)_{\theta_j}$, where $j$ enumerates elements of the $i$th row.

Let
\begin{equation}\label{Phi-12}
\begin{split}
\Phi_1:(y,\eta) & \mapsto\left(z^{(1)}(y,\eta),\zeta^{(1)}(y,\eta)\right),\\
\Phi_2:(x,\xi) &\mapsto\left(z^{(2)}(x,\xi),\zeta^{(2)}(x,\xi)\right),
\end{split}
\end{equation}
and 
\begin{equation}\label{Phi}
\Phi:=\Phi_2^{-1}\circ\Phi_1:(y,\eta)\mapsto(x^*(y,\eta),\xi^*(y,\eta))\,.
\end{equation}
Note that  
\begin{align}
(\zeta^{(1)}_\eta)^Tz^{(1)}_\eta-(z^{(1)}_\eta)^T\zeta^{(1)}_\eta\ &=\ (\zeta^{(2)}_\xi)^Tz^{(2)}_\xi-(z^{(2)}_\xi)^T\zeta^{(2)}_\xi\ =\ 0\,,\label{preserve-2a}
\\
(\zeta^{(1)}_\eta)^Tz^{(1)}_y-(z^{(1)}_\eta)^T\zeta^{(1)}_y\ &=\ (\zeta^{(2)}_\xi)^Tz^{(2)}_x-(z^{(2)}_\xi)^T\zeta^{(2)}_x\ =\ I\label{preserve-2b}
\end{align}
\begin{equation}\label{preserve-1a}
(z^{(1)}_\eta)^T\zeta^{(1)}\ =\ (z^{(2)}_\xi)^T\zeta^{(2)}\ =\ 0\,,
\end{equation}
\begin{equation}\label{preserve-1b}
(z^{(1)}_y)^T\zeta^{(1)}=\eta\quad\text{and}\quad (z^{(2)}_x)^T\zeta^{(2)}=\xi
\end{equation}
for all $(y,\eta)$ and $(x,\xi)$ in any local coordinates because the transformations $\Phi_j$ preserves the 2-form $\dr z\wedge\dr\zeta$ and the 1-form $\zeta\cdot\dr z$.

The proof proceeds in several steps.

\subsection{Part 1}\label{s:a1} Assume that  $\,\Phi_1(y_0,\eta_0)=\Phi_2(x_0,\xi_0):=(z_0,\zeta_0)\,$. Then there exists a local coordinate system $\,z=(z_1,\ldots,z_n)\,$ in a neighbourhood of $z_0$ such that $\det\zeta^{(1)}_\eta(y_0,\eta_0)\ne0$ and $\det\zeta^{(2)}_\xi(x_0,\xi_0)\ne0$. 

Indeed, let $\tilde z$ be arbitrary coordinates in a neighbourhood of $z_0$. One can easily show that  under a change of coordinates $\tilde z\to z$ the matrices $\,\zeta^{(1)}_\eta\,$ and $\,\zeta^{(2)}_\xi\,$ transform in the following way
\begin{equation}\label{xi-eta-1}
\begin{split}
\zeta^{(1)}_\eta\ &=\ \left.(\tilde z_z)^T\right|_{z=z^{(1)}}\tilde\zeta^{(1)}_\eta+C^{(1)}\left.(\tilde z_z)^{-1}\right|_{z=z^{(1)}}\tilde z^{(1)}_\eta\,,\\
\zeta^{(2)}_\xi\ &=\ \left.(\tilde z_z)^T\right|_{z=z^{(2)}}\tilde\zeta^{(2)}_\xi+C^{(2)}\left.(\tilde z_z)^{-1}\right|_{z=z^{(1)}}\tilde z^{(2)}_\xi\,,
\end{split}
\end{equation}
where $C^{(1)}$, $C^{(2)}$ are symmetric matrices with entries 
$$
C^{(1)}_{ik}=\sum_{m=1}^n\tilde\zeta_m^{(1)}\left.\frac{\partial^2\tilde z_m}{\partial z_i\,\partial z_k}\right|_{z=z^{(1)}},
\quad 
C^{(2)}_{ik}=\sum_{m=1}^n\tilde\zeta_m^{(2)}\left.\frac{\partial^2\tilde z_m}{\partial z_i\,\partial z_k}\right|_{z= z^{(2)}}
$$
(see, for instance, \cite[Section 2.3]{SV2}).
Clearly, one can choose coordinates $z$ in such a way that 
$$
\left.(\tilde z_z)\right|_{z=z^{(1)}(y_0,\eta_0)}=\ \left.(\tilde z_z)\right|_{z=z^{(2)}(x_0,\xi_0)}\ =\ I
$$
and $\,C^{(1)}(y_0,\eta_0)=C^{(2)}(x_0,\xi_0)\ =\ c\,I\,$, where $c$ is an arbitrary real constant. Then \eqref{xi-eta-1} turn into
\begin{equation}\label{xi-eta-2}
\begin{split}
\zeta^{(1)}_\eta(y_0,\eta_0)\ & =\ \tilde\zeta^{(1)}_\eta(y_0,\eta_0)+c\,\tilde z^{(1)}_\eta(y_0,\eta_0),\\
\zeta^{(2)}_\xi(x_0,\xi_0)\ &=\ \tilde\zeta^{(2)}_\xi(x_0,\xi_0)+c\,\tilde z^{(2)}_\xi(x_0,\xi_0)\,.
\end{split}
\end{equation}
In view of \eqref{preserve-2a}, $\tilde\zeta^{(1)}_\eta$ and $\tilde\zeta^{(2)}_\xi$ map the kernels of the matrices  $\tilde z^{(1)}_\eta$ and $\tilde z^{(2)}_\xi$ into the orthogonal complements of their ranges. Since $\tilde\zeta^{(1)}_\eta$ and $\tilde\zeta^{(2)}_\xi$ are both invertible this implies that the matrices on the right hand side of the equalities \eqref{xi-eta-2} are non-degenerate for sufficiently large $\,c\,$.

\subsection{Part 2}\label{s:a2}  

Let $\,(y_0,\eta_0)$, $(x_0,\xi_0)$ and $(z_0,\zeta_0)\,$ be as in Part 1, and let $z$ be an arbitrary local coordinate system such that $\det\zeta^{(1)}_\eta(y_0,\eta_0)\ne0$ and $\det\zeta^{(2)}_\xi(x_0,\xi_0)\ne0$. 

Assume first that $\conesupp V_1$ lies in a sufficiently small conic neighbourhood  $\mathcal O_1$ of the point $\,(y_0,\eta_0)\,$ such that $\,\det\zeta^{(1)}_\eta(y,\eta)\ne0$ for all $(y,\eta)\in\mathcal O_1\,$ and $\,\det\zeta^{(2)}_\xi(x,\xi)\ne0$ for all $(x,\xi)\in\Phi_1(\mathcal O_1)\,$. If 
\begin{equation}\label{phi}
\begin{split}
\varphi^{(1)}(z,y,\eta)\ &=\ (z-z^{(1)}(y,\eta))\cdot\zeta^{(1)}(y,\eta)\,,\\
\varphi^{(2)}(z,x,\xi)\ &=\ (z-z^{(2)}(x,\xi))\cdot\zeta^{(2)}(x,\xi)
\end{split}
\end{equation}
then, by \eqref{preserve-1a}, 
\begin{equation*}
\begin{split}
\varphi^{(1)}_\eta(z,y,\eta)\ &=\ (z-z^{(1)}(y,\eta))\cdot\zeta^{(1)}_\eta(y,\eta)\,,\\
\varphi^{(2)}_\xi(z,x,\xi)\ &=\ (z-z^{(2)}(x,\xi))\cdot\zeta^{(2)}_\xi(x,\xi)\,,
\end{split}
\end{equation*}
$\varphi^{(1)}_{z\eta}=\zeta^{(1)}_\eta\,$ and $\varphi^{(2)}_{z\xi}=\zeta^{(2)}_\xi\,$. Therefore the phase functions $\varphi^{(j)}$ satisfy the conditions {\bf(a$_3$)}--{\bf(a$_5$)} of Subsection \ref{FIO:def} for all $(y,\eta)\in\mathcal O_1\,$, $(x,\xi)\in\Phi_1(\mathcal O_1)\,$ and $z$ sufficiently closed to $z_0$.

Since ${\mathrm WF}\,(V_1u)\subset\Phi_1(\mathcal O_1)\,$ for all distributions $\,u\,$, we may assume without loss of generality that $\,\conesupp V_2\subset\Phi_1(\mathcal O_1)\,$. Then, in view of the above, the Schwartz kernels $\VC_1(z,y)$ and $\VC_2(z,x)$ of the FIO $V_j$ can be represented by oscillatory integrals 
\begin{equation}\label{composition0}
\begin{split}
(2\pi)^{-n}\int_{T^*_yM} e^{i\varphi^{(1)}(z,y,\eta)}p_1(y,\eta)\,
\left|\det\zeta^{(1)}_\eta\right|^{1/2}\,
\varsigma_1(z,y,\eta)\,\dr\eta\,,\\
(2\pi)^{-n}\int_{T^*_yM} e^{i\varphi^{(2)}(z,x,\xi)}p_2(x,\xi)\,
\left|\det\zeta^{(2)}_\xi\right|^{1/2}\,
\varsigma_2(z,x,\xi)\,\dr\xi
\end{split}
\end{equation}
of the form \eqref{integral} with the phase functions \eqref{phi}. The
Schwartz kernel of the composition $V_2^*V_1$ coincides with
\begin{multline}\label{composition1}
\iiint e^{i\psi(x,\xi,z,y,\eta)}b(x,\xi,z,y,\eta)\,p_1(y,\eta)\,
\overline{p_2(x,\xi)}\,\dr\eta\,\dr z\,\dr\xi\\
= \iiint|\eta|^{-n} e^{i\psi(x,|\eta|\xi,z,y,\eta)}b(x,|\eta|\xi,z,y,\eta)\,
p_1(y,\eta)\,\overline{p_2(x,|\eta|\xi)}\,\dr\eta\,\dr z\,\dr\xi\,,
\end{multline}
where the integrals are taken over $T^*_xM\times M\times T^*_yM$,
$$
\psi(x,\xi,z,y,\eta)\ =\ (z-z^{(1)}(y,\eta))\cdot\zeta^{(1)}(y,\eta)-(z-z^{(2)}(x,\xi))
\cdot\zeta^{(2)}(x,\xi)
$$
and
\begin{multline*}
b(x,\xi,z,y,\eta)\\
=(2\pi)^{-2n}\left|\det\zeta^{(1)}_\eta(y,\eta)\right|^{1/2}
\left|\det\zeta^{(2)}_\xi(x,\xi)\right|^{1/2}\,
\varsigma_1(z,y,\eta)\,\varsigma_2(z,x,\xi)\,.
\end{multline*}

Now we are going to apply the stationary phase method to the integral with respect to the variables $z$ and $\xi$, considering $|\eta|$ as a large parameter. A rigorous justification of the stationary phase formula for non-convergent integrals of this type can be found, for instance, in \cite[Appendix C]{SV2}. 

The equations $\psi_\xi=0$ and $\psi_z=0$ are equivalent to
\begin{equation}\label{st-phase1}
z=z^{(2)}(x,\xi)\quad\text{and}\quad\zeta^{(2)}(x,\xi)=\zeta^{(1)}(y,\eta)
\end{equation}
respectively.
Since $\zeta^{(2)}(x_0,\xi_0)=\zeta^{(1)}(y_0,\eta_0)$ and $\det\zeta^{(2)}_\xi(x_0,\xi_0)\ne0$, in a neighbourhood of $(x_0,y_0,\eta_0)$, the second equation \eqref{st-phase1} has a unique $\xi$-solution $\hat\xi(x,y,\eta)$ 
such that $\hat\xi(x_0,y_0,\eta_0)=\xi_0$. Thus the stationary point is $(z,\xi)=(z^{(2)}(x,\hat\xi),\hat\xi)$. 
It is unique and non-degenerate because
\begin{equation}\label{hess}
\begin{pmatrix}
\psi_{zz}&\psi_{z\xi}\\ \psi_{\xi z}&\psi_{\xi\xi}
\end{pmatrix}
\ =\
\begin{pmatrix}
0&\zeta^{(2)}_\xi\\ \left(\zeta^{(2)}_\xi\right)^T&\psi_{\xi\xi}
\end{pmatrix}.
\end{equation}

By the stationary phase formula, the integral \eqref{composition1} coincides modulo a smooth function with
\begin{equation}\label{composition2}
(2\pi)^{-n}\int e^{i\varphi(x,y,\eta)}\left|\det\zeta^{(2)}_\xi(x,\hat\xi) \right|^{-1}\,\tilde p(x,y,\eta)\,\dr\eta\,,
\end{equation}
where
\begin{multline}\label{composition3}
\varphi(x,y,\eta):=
\psi(x,\hat\xi,z^{(2)}(x,\hat\xi),y,\eta)\\ 
=\ \left(z^{(2)}(x,\hat\xi)-z^{(1)}(y,\eta)\right)\cdot\zeta^{(1)}(y,\eta)
\end{multline}
and $\tilde p$ is an amplitude of class $S^{m_1+m_2}_{\mathrm{phg}}$ with the leading homogeneous term
$$
(2\pi)^{2n}\,b(x,\hat\xi,z^{(2)}(x,\hat\xi),y,\eta)\;
p_1(y,\eta)\,\overline{p_2(x,\hat\xi)}
$$
such that
$$
\conesupp\tilde p\subset\{(x,y,\eta)\,:\,(y,\eta)\in\conesupp p_1\,,\,(x,\hat\xi)\in\conesupp p_2\}
$$
(we have used the fact that the signature of the Hessian \eqref{hess} is equal to zero).

Clearly, $\hat\xi(x^*(y,\eta),y,\eta)=\xi^*(y,\eta)$ and $z^{(2)}(x^*,\xi^*)=z^{(1)}(y,\eta)$. Thus $\varphi(x^*,y,\eta)=0$. Since
$$
\psi_z(x,\hat\xi,z^{(2)}(x,\hat\xi),y,\eta)
=\psi_\xi(x,\hat\xi,z^{(2)}(x,\hat\xi),y,\eta)=0
$$
for all $x,y,\eta$, we also have
$$
\varphi_x(x,y,\eta)=\psi_x(x,\hat\xi,z^{(2)}(x,\hat\xi),y,\eta)
=(z^{(2)}_x(x,\hat\xi))^T\,\zeta^{(2)}(x,\hat\xi)\,.
$$
Now the second equality \eqref{preserve-1b} implies that
$\varphi_x(x,y,\eta)=\hat\xi(x,y,\eta)$. Substituting $x=x^*$, we obtain
\begin{equation}\label{x}
\varphi_x(x^*,y,\eta)=(z^{(2)}_x(x^*,\xi^*))^T\,\zeta^{(2)}(x^*,\xi^*)=\xi^*\,.
\end{equation}

Similarly,
\begin{multline*}
\varphi_\eta(x,y,\eta)=\psi_\eta(x,\hat\xi,z^{(2)}(x,\hat\xi),y,\eta)\\
=\left.\nabla_\eta\left((z-z^{(1)}(y,\eta))\cdot
\zeta^{(1)}(y,\eta)\right)\right|_{z=z^{(2)}(x,\hat\xi)}\,.
\end{multline*}
This equality and \eqref{preserve-1a} imply that
$$
\varphi_\eta(x,y,\eta)=(z^{(2)}(x,\hat\xi)-z^{(1)}(y,\eta))
\cdot\zeta^{(1)}_\eta(y,\eta)\,,
$$
and, consequently,
$$
\varphi_{x\eta}(x,y,\eta)=\left(\nabla_x\,z^{(2)}(x,\hat\xi)\right)^T
\zeta^{(1)}_\eta(y,\eta)\,.
$$
Differentiating the identity $\zeta^{(2)}(x,\hat\xi)\equiv\zeta^{(1)}(y,\eta)$, we obtain
$$
\zeta^{(2)}_\xi(x,\hat\xi)\,\hat\xi_x\ =\ -\zeta^{(2)}_x(x,\hat\xi)\,.
$$
From the above two equalities it follows that
$$
\varphi_{x\eta}(x^*,y,\eta)\ =\
\left(z^{(2)}_x-z^{(2)}_\xi\,\left(\zeta^{(2)}_\xi\right)^{-1}
\zeta^{(2)}_x\right)^T\zeta^{(1)}_\eta\,,
$$
where $z^{(1)}_\eta=z^{(1)}_\eta(y,\eta)$, $\zeta^{(1)}_\eta=\zeta^{(1)}_\eta(y,\eta)$, and $z^{(2)}_x$, $z^{(2)}_\eta$, $\zeta^{(2)}_x$, $\zeta^{(2)}_\eta$ are evaluated at $(x^*,\xi^*)$.

In view of \eqref{preserve-2a} and \eqref{preserve-2b},
\begin{multline*}
z^{(2)}_x-z^{(2)}_\xi\,\left(\zeta^{(2)}_\xi\right)^{-1}
\zeta^{(2)}_x\\
=\left((\zeta^{(2)}_\xi)^T\right)^{-1}
\left((\zeta^{(2)}_\xi)^Tz^{(2)}_x
-(z^{(2)}_\xi)^T\zeta^{(2)}_x\right)=\left((\zeta^{(2)}_\xi)^T\right)^{-1}
\end{multline*}
Consequently,
\begin{equation}\label{x-eta1}
\varphi_{x\eta}(x^*,y,\eta)\ =\
\left(\zeta^{(2)}_\xi(x^*,\xi^*)\right)^{-1}\,
\zeta^{(1)}_\eta(y,\eta)
\end{equation}
and
\begin{equation}\label{x-eta2}
\left|\det\varphi_{x\eta}(x^*,y,\eta)\right|^{1/2}\ =\ \left|\det\zeta^{(1)}_\eta(y,\eta)\right|^{1/2}
\left|\det\zeta^{(2)}_\xi(x^*,\xi^*)\right|^{-1/2}.
\end{equation}

Since $\det\zeta^{(j)}_\eta\ne0$, the above equality and \eqref{x} imply that the phase function \eqref{composition3} satisfies the conditions of Subsection \ref{FIO:def}. This shows that \eqref{composition2} defines the Schwartz kernel of a FIO associated with the canonical transformation $\Phi$. Applying the procedure described in Remark \ref{r:fio-def3}, we can remove the dependence on $x$ and rewrite it in the form
\begin{equation}\label{composition4}
(2\pi)^{-n}\int e^{i\varphi(x,y,\eta)}\,p(y,\eta)\,
\left|\det\varphi_{x\eta}(x^*,y,\eta)\right|^{1/2}\,\varsigma(x,y,\eta)\,\dr\eta
\end{equation}
where $\varsigma$ is a cut-off function satisfying the conditions of Subsection \ref{FIO:def} and
$p(y,\eta)$ is an amplitude of class $S^{m_1+m_2}_{\mathrm{phg}}$ such that
$$
\conesupp p\subset\{(y,\eta)\in\conesupp p_1\,:\,(x^*,\xi^*)\in\conesupp p_2\}\,.
$$
Clearly, the leading homogeneous term $p_0$ of the amplitude $p$ is given by the formula
\begin{equation}\label{amplitudes}
p_0(y,\eta)\  =\  p_{1,0}(y,\eta)\,\overline{p_{2,0}\left(\Phi(y,\eta)\right)},
\end{equation}
where $p_{j,0}$ are the leading homogeneous terms of the amplitudes $p_j$.

\subsection{Part 3}\label{s:a3}

Consider now general FIOs $V_j$ associated with the canonical transformations $\Phi_j$. Splitting $V_j$ into sums of FIOs with simply connected conic supports,  we see that it is sufficient to prove the theorem assuming that the bundles $\DC_{\mathbb{Z}}(\Phi_j)$ are topologically trivial and the Schwartz kernels of $V_j$ are given by oscillatory integrals $\IC_j$ of the form \eqref{integral1} with phase functions $\varphi_j$ and full symbols $q_j$.

If the conic supports of $q_j$ are sufficiently small then, choosing suitable local coordinates and transforming $\varphi_j$ into the phase functions \eqref{phi}, we can rewrite the corresponding oscillatory integrals in the form \eqref{composition0} with $p_j=i^{k_j}q_j$, where $k_j$ is an integer determined by the choice of branch of $\arg({\det}^2(\varphi_j)_{x\eta})$. In this case, by Part 2, the composition $V_2^*V_1$ is a FIO given by the oscillatory integral \eqref{composition4} with the local phase function \eqref{composition3} and an amplitude $p\in S^{m_1+m_2}_{\mathrm{phg}}$ with the leading homogeneous term
$$
p_0(y,\eta)\  =\  i^{k_1-k_2}\,q_{1,0}(y,\eta)\,\overline{q_{2,0}\left(\Phi(y,\eta)\right)}\,.
$$

Let $\tilde\varphi$ be an arbitrary global phase function associated with the transformation $\Phi$. Since $\DC(\Phi)=\DC(\Phi_1)$ is simply connected,  the bundle $\DC_{\mathbb{Z}}(\Phi)$ is also trivial. Let us fix a continuous branch of $\arg({\det}^2(\tilde\varphi_{x\eta})$. Transforming the phase function $\varphi$ given by \eqref{composition3} into $\tilde\varphi$, we see that  \eqref{composition4} coincides with an oscillatory integral of the form
 $$
(2\pi)^{-n}\int_{T^*_yM} e^{i\tilde\varphi(x,y,\eta)}i^{-k}\,\tilde q(y,\eta)\,\left(\mathrm{det}^2\,\tilde\varphi_{x\eta}(x,y,\eta)\right)^{1/4}\,
\varsigma(x,y,\eta)\,d\eta\,,
$$
where $\tilde q\in S^{m_1+m_2}_{\mathrm{phg}}$ is another amplitude with the same leading homogeneous term $i^{k_1-k_2}\,q_{1,0}(y,\eta)\,\overline{q_{2,0}\left(\Phi(y,\eta)\right)}$ and $k$ is the integer such that $\frac{k\pi}2=\arg({\det}^2\varphi_{x\eta})$ turns into $\arg({\det}^2\tilde\varphi_{x\eta})$ under continuous transformation of the phase functions $\phi\mapsto\tilde\varphi$ (see Remark \ref{r:proof2}).

Thus we have proved that, for $V_j$ with small conic supports, the composition $V_2^*V_1$ is a FIOs of order $m_1+m_2$ with principal symbol 
$$
i^{k_1-k_2-k}\sigma_{V_1}(y,\eta)\,\sigma_{V_2}(\Phi(y,\eta))
$$ 
such that
$$
\conesupp(V_2^*V_1)\subset\left(\conesupp V_1\,\bigcap\,\Phi^{-1}(\conesupp V_2)\right).
$$
Obviously, the integer $k_1-k_2-k$ is uniquely defined by the choice of branches of $\arg({\det}^2(\varphi_j)_{x\eta})$ and $\arg({\det}^2(\tilde\varphi_{x\eta})$. Therefore, using a partition of unity on $T^*M$, we see that the same result holds for all FIOs $V_j$. Since the principal symbols are defined modulo a factor $i^m$ with an integer $m$, this completes the proof.
\qed

\section{Sketch of proof of Lemma \ref{l:split}$\mathrm{(i)}$ and $\mathrm{(ii)}$}\label{s:b}

Let the FIOs corresponding to the incoming, reflected and refracted trajectories are given by the oscillatory integrals \eqref{integral1} with phase functions $\varphi$, $\varphi^+$, $\varphi^-$ and symbols $q$, $q^+$ and $q^-$ respectively. The first two are standard oscillatory integrals defined in Section \ref{s:FIO}. In the case (i), the third is also a standard oscillatory integral. In the case (ii), it is a boundary layer oscillatory integral given by the same expression \eqref{integral1} but with a complex-valued phase function satisfying the conditions of \cite[Section 2.6.4]{SV2}. 

Since we are dealing with half-densities, elements $u$ in the domain of the operator satisfy
\begin{gather*}
 \mathbf{h}_+^{-1/2}(x) u_+(x) = \mathbf{h}_-^{-1/2}(x) u_-(x),\\
 \mathbf{h}_+(x) \partial_n^+ (\mathbf{g}^{-1/2}  u)(x) = -\mathbf{h}_-(x) \partial_n^- (\mathbf{g}^{-1/2} u)(x)
\end{gather*}
in each local coordinate chart for each $x \in Y$ in the domain of the chart, where $u_\pm$ are the left and right boundary values of $u$
and $\partial_n^\pm$ denote the inward $g_\pm$-normal derivatives. 

Substituting the sum of the integrals into the boundary conditions and equating to zero the sum of leading terms at $t=t^*$ and $x=x^*$, we obtain the following equations, 
\begin{equation}\label{b:phase}
\varphi(t^*,x^*,y,\eta)\ =\  \varphi^+(t^*,x^*,y,\eta)\ =\ \varphi^-(t^*,x^*,y,\eta)\,,
\end{equation}
\begin{multline}\label{b:symbols1}
\mathbf h_+^{-1/2}(x) \left.\left(u_{\kappa,j}\,d_\varphi+ u_{\kappa_0,j+1}\,d_{\varphi^+}\right)\right|_{t=t^*,\,x=x^*} \\ \ =\  \mathbf h_-^{-1/2}(x) \left.\left(u_{\kappa_1,j+1}\,d_{\varphi^-}\right)\right|_{t=t^*,\,x=x^*}
\end{multline}
and 
\begin{multline}\label{b:symbols2}
\mathbf h_+^{1/2}(x) \left.\left(u_{\kappa,j}\,d_\varphi\,\partial_n^+\varphi+ u_{\kappa_0,j+1}\,d_{\varphi^+}\,\partial_n^+\varphi^+\right)\right|_{t=t^*,\,x=x^*}\\ 
=\  -\mathbf h_-^{1/2}(x) \left.\left(u_{\kappa_1,j+1}\,d_{\varphi^-}\,\partial_n^-\varphi^-\right)\right|_{t=t^*,\,x=x^*}
\end{multline}
where $d_\psi:=\left(\det^2\psi_{x\eta}\right)^{1/4}.$

Then condition {\bf(a$_4$)} implies that 
$$
-\,\left.\partial_n^+\varphi\right|_{t=t^*,\,x=x^*}=\left.\partial_n^+\varphi^+\right|_{t=t^*,\,x=x^*}=\tau_+\,.
$$
Similarly, in the case (i), $\,\left.\partial_n^-\varphi^-\right|_{t=t^*,\,x=x^*}=\tau_-\,$. In the case (ii), by \cite[(2.6.23)]{SV2}, we have
$\,\left.\partial_n^-\varphi^-\right|_{t=t^*,\,x=x^*}=i\,\tilde\tau_-\,$.

From the equalities \eqref{b:phase}, \cite[(2.5.3$^\pm$)]{SV2} and \cite[(2.6.25)]{SV2} it follows that at the point $t=t^*$, $x=x^*$
\begin{equation*}
-\tau_+\,\det\varphi_{x\eta}\ =\ \tau_+\,\det\varphi^+_{x\eta}\ =
\begin{cases}
\tau_-\,\det\varphi^-_{x\eta}\ &\text{in the case (i)},\\
i\,\tilde\tau_-\,\det\varphi^-_{x\eta}&\text{in the case (ii)}.
\end{cases}
\end{equation*}
(cf. \cite[(2.6.14)]{SV2}). Consequently, $\left.d_\varphi\right|_{t=t^*,\,x=x^*}=\left.d_{\varphi^+}\right|_{t=t^*,\,x=x^*}$ and 
\begin{equation*}
\left.d_{\varphi^-}\right|_{t=t^*,\,x=x^*}=
\begin{cases}
\sqrt{\tau_+/\tau_-}\;d_\varphi\ &\text{in the case (i)},\\
\left(\tau_+/i\,\tilde\tau_-\right)^{1/2}\,\;d_\varphi&\text{in the case (ii)},
\end{cases}
\end{equation*}
where $\left(\tau_+/i\,\tilde\tau_-\right)^{1/2}$ is a continuous branch of the square root.

In view of the above equalities, the equations \eqref{b:symbols1}, \eqref{b:symbols2} imply that 
\begin{eqnarray*}
b_+^{-1/2} \left( u_{\kappa,j}+ u_{\kappa_0,j+1} \right) &=& \sqrt{\tau_+/\tau_-}\;u_{\kappa_1,j+1}\,,\\
b_+^{1/2} \left( -\tau_+\,u_{\kappa,j}+ \tau_+\,u_{\kappa_0,j+1} \right)
&=&  -\,\sqrt{\tau_+\,\tau_-}\;u_{\kappa_1,j+1}\,,
\end{eqnarray*}
in the case (i), and 
\begin{eqnarray*}
b_+^{-1/2} \left(  u_{\kappa,j}+ u_{\kappa_0,j+1} \right) &=& \left(\tau_+/i\,\tilde\tau_-\right)^{1/2}\;u_{\kappa_1,j+1}\,,\\
b_+^{1/2} \left( -\tau_+\,u_{\kappa,j}+ \tau_+\,u_{\kappa_0,j+1} \right)
&=&  -i\tilde\tau_-\left(\tau_+/i\,\tilde\tau_-\right)^{1/2}\;u_{\kappa_1,j+1}
\end{eqnarray*}
in the case (ii), where we have used $b_+(x) = \mathbf h _+(x) \mathbf h _-^{-1}(x)$. Solving these equations with respect to $\,u_{\kappa_0,j+1}\,$,
$\,u_{\kappa_1,j+1}\,$, we obtain the required formulae for the coefficients $\,\tau_{\kappa_0,j+1}\,$,
$\,\tau_{\kappa_1,j+1}\,$.

\section{Proofs of Theorem \ref{t:local-fio} and Lemma \ref{l:thetas}}\label{s:c}

\subsection{Proof of Theorem \ref{t:local-fio}}

Without loss of generality we shall be assuming that $\conesupp V$ lies in a sufficiently small neighbourhood of a fixed point $(y_0,\eta_0)\in T^*(M\setminus\partial M)$. 

Let $t\in(-\delta,\delta)$ with a sufficiently small $\delta$. If $\delta$ is smaller than the geodesic distance from the support of the Schwartz kernel of $V$ to the boundary $\partial M$ then, 
in view of Theorems \ref{t:parametrix}, Corollary \ref{c:composition} and Remark \ref{r:negative-time},  $VU(t)$ is a FIO associated with the canonical transformation $\Phi_V^t:=\Phi\circ\Phi^t$ with principal symbol $\sigma_{VU(t)}(y,\eta)=\sigma_V(\Phi^t(y,\eta))$,
where 
$$
\Phi^t:(y,\eta)\ \mapsto\ (x^t(y,\eta),\xi^t(y,\eta))
$$
is the shift along geodesics in $M\setminus\partial M$. Denote 
$$
\Phi(y,\eta):=\left(x^*(y,\eta),\xi^*(y,\eta)\right)\,,\qquad \Phi_V^t(y,\eta):=\left(z^t(y,\eta),\zeta^t(y,\eta)\right).
$$

Assume that $\delta$ and $\conesupp V$ are small enough. Then the union $\bigcup_{t\in(-\delta,\delta)}\conesupp V(t)$ is also small and, for $t\in(-\delta,\delta)$, the Schwartz kernel of $VU(t)$ can be represented (modulo smoothing operators) by an oscillatory integral 
\begin{equation}\label{fio-1}
(2\pi)^{-n}\int_{T^*_yM} e^{i\varphi(t;x,y,\eta)}p(t;y,\eta)\,
\left|\det\zeta^t_\eta(y,\eta)\right|^{1/2}
\varsigma(t;x,y,\eta)\,\dr\eta\,,
\end{equation}
with a phase function given by the equality
\begin{equation}\label{varphi-1}
\varphi(t,x,y,\eta)\ =\ (x-z^t(y,\eta))\cdot\zeta^t(y,\eta)
\end{equation}
in a local coordinate system such that $\det\zeta_\eta^t(y,\eta)\ne0$, an amplitude $p$ with small conic support, and a cut-off function $\varsigma$ satisfying the conditions of Subsection \ref{FIO:def} (see Appendix \ref{s:a}).

\begin{lem}\label{l:fio-local}
Let $\Sigma_\delta(\Phi)$ be the set of points $(y,\tilde\eta)\in S^*M$ such that $\Phi_V^{t^\star}(y,\tilde\eta)=(y,\tilde\eta)$ at some time $t^\star=t^\star(y,\tilde\eta)\in(-\delta,\delta)$.
If the Schwartz kernel of $VU(t)$ is given by \eqref{fio-1} and $\hat\rho\in C_0^\infty(-\delta,\delta)$  then
\begin{multline}\label{fio-2}
\FC^{-1}_{t\to\lambda}\left(\hat\rho(t)\,\Tr\,(VU(t))\right)\\
=(2\pi)^{-n}\lambda^{n-1}\int_{\Sigma_\delta(\Phi)}e^{i\lambda t^\star}\hat\rho(t^\star)\,
p_0(t^\star;y,\tilde\eta)\,
\dr y\,\dr\tilde\eta+o(\lambda^{n-1})
\end{multline}
as $\lambda\to+\infty$, where $t^\star=t^\star(y,\tilde\eta)$ and $p_0$ is the leading homogeneous term of the amplitude $p$.
\end{lem}

\begin{proof}

Clearly, $z^t=x^*\left(x^t(y,\eta),\xi^t(y,\eta)\right)$. Therefore
$$
\frac{\dr}{\dr t}z^t\ =\ x^*_y(x^t,\xi^t)\,\frac{\dr}{\dr t}x^t\;+\;
x^*_\eta(x^t,\xi^t)\,\frac{\dr}{\dr t}\xi^t\,.
$$
Since $\Phi$ preserves the 1-form $\xi\cdot\dr x$, we have
\begin{eqnarray*}
\left(x^*_y(x^t,\xi^t)\right)^T\zeta^t &=& \left(x^*_y(x^t,\xi^t)\right)^T \xi^*(x^t,\xi^t)\ =\ \xi^t\,,\\
\left(x^*_\eta(x^t,\xi^t)\right)^T\zeta^t &=& \left(x^*_\eta(x^t,\xi^t)\right)^T\xi^*(x^t,\xi^t)\ =\ 0
\end{eqnarray*}
and, consequently, $\zeta^t\cdot\frac{\dr}{\dr t}z^t=\xi^t\cdot\frac{\dr}{\dr t}x^t$. By Euler's identity for homogeneous functions, 
$\,\xi^t\cdot\frac{\dr}{\dr t}x^t=\xi^t\cdot h_\xi(x^t,\xi^t)=h(x^t,\xi^t)\,$
where $h(x,\xi)=\sqrt{g(x,\xi)}$. Now, differentiating \eqref{varphi-1}, we see that
\begin{equation}\label{varphi-2}
\varphi_t(t;x,y,\eta)\ =\ -\sqrt{g(x^t,\xi^t)}\,+\,(x-z^t(y,\eta))\cdot\frac{\dr}{\dr t}\,\zeta^t(y,\eta)\ \ne\ 0
\end{equation}
for all $(t,x,y,\eta)\in\supp\varsigma$ provided that $\supp\varsigma$ is small enough.

Let
\begin{equation}\label{fio-p}
\tilde p(t;y,\eta)\ =\ \hat\rho(t)\,p(t;y,\eta)\left|\det\zeta^t_\eta(y,\eta)\right|^{1/2}
\varsigma(t,y,y,\eta)\,.
\end{equation}
Then
\begin{multline*}
\FC^{-1}_{t\to\lambda}\left(\hat\rho(t)\,\Tr\,(VU(t))\right)\\
=\ 
(2\pi)^{-n-1}\int\int_{T^*_yM} e^{i(\varphi(t;y,y,\eta)+\lambda t)}\,\tilde p(t;y,\eta)\,\dr y\,\dr\eta\,\dr t\, + O(\lambda^{-\infty}).
\end{multline*}
Changing variables $\eta=\lambda r\tilde\eta$, where $r\in[0,+\infty)$ and $\tilde\eta$ are coordinates on the cosphere $S^*_yM$, we obtain
\begin{multline}\label{fio-3}
\FC^{-1}_{t\to\lambda}\left(\hat\rho(t)\,\Tr\,(VU(t)\Op(\chi))\right)\\
=\ 
(2\pi)^{-n-1}\lambda^n\iint\int_{S^*_yM} e^{i\lambda\psi(t;r,y,\tilde\eta)}\,\tilde p(t;y,\lambda r\tilde\eta)\,
\dr y\,\dr\tilde\eta\,\dr r\,\dr t + O(\lambda^{-\infty}),
\end{multline}
with 
$$
\psi(t;r,y,\tilde\eta)\ =\ r\varphi(t;y,y,\tilde\eta)+t\,.
$$

Now we are going to apply the stationary phase formula with respect to $t$ and $r$ (see \cite[Appendix C]{SV2} for justification of this procedure). In view of \eqref{varphi-2},
$$
\det\,\begin{pmatrix}
\psi_{rr} & \psi_{rt}\\ \psi_{tr} & \psi_{tt}
\end{pmatrix}
\ =\ -(\varphi_t)^2\ \ne\ 0
$$
on $\supp\tilde p$. Thus all stationary points of $\psi$ lying in $\supp\tilde p$ are non-degenerate. They are given by the equations
\begin{equation}\label{st-points}
\varphi(t;y,y,\tilde\eta)=0 \quad\text{and}\quad r\varphi_t(t;y,y,\tilde\eta)+1=0\,.
\end{equation}
By \eqref{varphi-2}, the first equation either does not have any solutions, or has a unique solution $t^\star(y,\tilde\eta)$ smoothly depending on $(y,\tilde\eta)$. Here we have used the implicit function theorem and assumed that $\delta$ is small enough so that $t^\star(y,\tilde\eta)$ is defined on
$\supp\tilde p$.

In the first case \eqref{fio-3} is a rapidly decreasing function of $\lambda$ and
$z^t(y,\eta)\ne y$ for all $(t;y,\eta)\in\supp\tilde  p$, which implies that $\Phi_V^t(y,\eta)\ne(y,\eta)$ for all $(y,\eta)\in\supp p$ and $t\in \supp\hat\rho$.

In the second case, by the stationary phase formula, 
\begin{multline}\label{fio-4}
\FC^{-1}_{t\to\lambda}\left(\hat\rho(t)\,\Tr\,(VU(t))\right)\\
=\ 
(2\pi)^{-n}\lambda^{n-1}\int_{S^*_yM} e^{i\lambda t^\star}r^*\,\tilde p(t^\star;y,\lambda r^*\tilde\eta)\,
\dr y\,\dr\tilde\eta+o(\lambda^{n-1})
\end{multline}
as $\lambda\to+\infty$, where $r^*=r^*(y,\tilde\eta):=\left|\varphi_t(t^\star;y,y,\tilde\eta)\right|^{-1}$. One can easily show that the right hand side of \eqref{fio-4} coincides with
\begin{equation}\label{fio-5}
(2\pi)^{-n}\lambda^{n-1}\int_\Omega e^{i\lambda t^\star}r^*\,\tilde p(t^\star;y,\lambda r^*\tilde\eta)\,
\dr y\,\dr\tilde\eta+o(\lambda^{n-1})\,,
\end{equation}
where 
$$
\Omega\ =\ \{(y,\tilde\eta)\in S^*M\,:\,t^\star_y(y,\tilde\eta)=t^\star_\eta(y,\tilde\eta)=0\}
$$
is the set of stationary points of the function $t^\star$. 

Differentiating the identity $\varphi(t^\star,y,y,\tilde\eta)\equiv0$ and taking into account \eqref{varphi-2}, we see that $\Omega$ consists of the points $(y,\tilde\eta)$ such that
$$
\varphi_\eta(t^\star,y,y,\tilde\eta)\ =\ (y-z^{t^\star}(y,\tilde\eta))\cdot\zeta_\eta^{t^\star}(y,\tilde\eta)\ =\ 0
$$
and
$$
\varphi_y(t^\star,y,y,\tilde\eta)\ =\ (I-z_y^{t^\star}(y,\tilde\eta))^T\zeta^{t^\star}(y,\tilde\eta)+(y-z^{t^\star}(y,\tilde\eta))\cdot\zeta_y^{t^\star}(y,\tilde\eta)\ =\ 0\,.
$$
The first equation implies that $y=z^{t^\star}(y,\eta)$. Since  $\Phi_V^t$ preserves the 1-form $\xi\cdot\dr x$, the vector-function $z_y^{t^\star}(y,\eta))^T\zeta^{t^\star}(y,\eta)$ in the second equation is identically equal to $\tilde\eta$. Therefore the second equation yields $\zeta^{t^\star}(y,\tilde\eta)=\tilde\eta$. 

Thus we see that $(y,\tilde\eta)\in\Omega$ if and only if $\Phi_V^{t^\star}(y,\tilde\eta)=(y,\tilde\eta)$. From here and \eqref{varphi-2} it follows that $r^*(y,\tilde\eta)=1$ for all $(y,\tilde\eta)\in\Omega$. Recall that 
\begin{equation}\label{zeros}
\begin{split} 
&\text{\sl for every $C^\infty$-function the set of zeros of infinite} \\
&\text{\sl order has full measure in the set of all its zeros.}
\end{split}
\end{equation}
In particular, this implies that $\nabla_\eta(\zeta^{t^\star}(y,\eta)-\tilde\eta)=\zeta_\eta^{t^\star}(y,\eta)-I=0$ 
and henceforth $\det \zeta_\eta^{t^\star}=1$
on a set of full measure in $\Omega$. Now, substituting \eqref{fio-p} into \eqref{fio-5} and removing lower order terms of $p$, we obtain the required result.
\end{proof}

Let $\Op(\chi)$ be the operator of multiplication by an arbitrary function $\chi\in C_0^\infty(M\setminus\partial M)$ such that $V=\Op(\chi)\,V=V\,\Op(\chi)$. Then $\sum_{m=0}^3 i^m\,V_m^*V_m=V$ and, consequently,
$\,N_V(\lambda)=\sum_{m=0}^3 i^m\,N_{V^*_mV_m}(\lambda)\,$, where $V_m:=\frac12\left(V+i^m\Op(\chi)\right)$. Note that $V^*_mV_m$ are linear combinations of FIOs lying in $\AC_0$. Applying Lemma \ref{l:fio-local} to these FIOs and taking into account \eqref{rho2}, we see that 
$$
\rho*N'_{V^*_mV_m}(\lambda)
=\FC^{-1}_{t\to\lambda}\left(\hat\rho(t)\,\Tr \left(V^*_mV_mU(t)\right)\right)
=O(\lambda^{n-1})\,,\qquad\lambda\to+\infty\,,
$$
whenever $\supp\rho\subset(-\delta,\delta)$ for a sufficiently small $\delta$. Assume, in addition, that the function $\rho$ is even and nonnegative, and $\hat\rho(0)=1$. Then, since the functions $N_{V^*_mV_m}$ are nondecreasing, standard Tauberian theorems imply that 
$$
\left|N_{V^*_mV_m}(\lambda)-\rho*N_{V^*_mV_m}(\lambda)\right|\ =\ O(\lambda^{n-1})\,,\qquad\lambda\to+\infty\,,
$$
(see, for instance, \cite[Theorem 1.3]{Sa2}).

Clearly, the same estimate holds for the function $N_V$. 
Applying \eqref{rho2} and integrating \eqref{fio-2}, we obtain
\begin{multline}\label{fio-6}
N_V(\lambda)\ =\ \rho*N_V(\lambda)\ =\ 
\int_0^\lambda\FC^{-1}_{t\to\mu}\left(\hat\rho(t)\,\Tr\,(VU(t))\right)\dr\mu\\
=\ (2\pi)^{-n}\int_0^\lambda\mu^{n-1}\int_{\Sigma_\delta(\Phi)}e^{i\mu t^\star}\hat\rho(t^\star)\,
p_0(t^\star;y,\tilde\eta)\,
\dr y\,\dr\tilde\eta\,\dr\mu+o(\lambda^n)\,.
\end{multline}
If $\eps>0$ and $\Sigma_{\delta,\eps}(\Phi):=\{(y,\tilde\eta)\in\Sigma_\delta(\Phi) \,:\, |t^\star(y,\tilde\eta)|>\eps\}$ then
$$
\int_0^\lambda\mu^{n-1}\int_{\Sigma_{\delta,\eps}(\Phi)}e^{i\mu t^\star}\hat\rho(t^\star)\,
p_0(t^\star;y,\tilde\eta)\,
\dr y\,\dr\tilde\eta\,\dr\mu\ =\ O(\lambda^{n-1})\,.
$$
Since 
$\,\Sigma_\delta(\Phi)\setminus\left(\bigcup_{\eps>0}\Sigma_{\delta,\eps}(\Phi)\right)=\Sigma_0(\Phi)=\mathrm{Fix}(\Phi)\,$,
letting $\eps\to0$ we see that the right hand side of \eqref{fio-6} coincides with
\begin{multline*}
 (2\pi)^{-n}\int_0^\lambda\mu^{n-1}\int_{\mathrm{Fix}(\Phi)}
p_0(0;y,\tilde\eta)\,
\dr y\,\dr\tilde\eta\,\dr\mu\;+\;o(\lambda^n)\\
 =\  (2\pi)^{-n}n^{-1}\lambda^n\int_{\mathrm{Fix}(\Phi)}p_0(0;y,\tilde\eta)\,\dr y\,\dr\tilde\eta\;+\;o(\lambda^n)\,.
\end{multline*}

In view of \eqref{zeros}, 
$\,x^*_\eta=\left.\varphi_{\eta\eta}\right|_{x=x^*}=0\,$
on a set of full measure in $\mathrm{Fix}(\Phi)$, where $\varphi$ is the phase function \eqref{varphi-1}. Now \eqref{sigma-theta}
implies that $p_0(0;y,\tilde\eta)=i^{\Theta_\Phi(y,\tilde\eta)}\sigma_V(y,\tilde\eta)$ on a set of full measure, and \eqref{Lambda}
follows from the Weyl asymptotic formula for the counting function $N(\lambda)$.

\subsection{Proof of Lemma \ref{l:thetas}}
Let $\varphi^{(j)}$ and $\varphi$ be arbitrary functions associated (in the sense of Subsection \ref{FIO:def}) with the transformations \eqref{Phi-12} and \eqref{Phi}.
The condition {\bf(a$_4$)} and \eqref{zeros} imply that
\begin{equation}\label{thetas-1}
x^*_\eta=z^{(1)}_\eta-z^{(2)}_\eta=0\quad\text{and}\quad \left.\varphi_{\eta\eta}\right|_{x=x^*}=\left.\left(\varphi^{(1)}_{\eta\eta}-\varphi^{(2)}_{\eta\eta}\right)\right|_{x=x^*}=0
\end{equation}
on a set of full measure in $\mathrm{Fix}(\Phi)$.

Let $V_j$ and $V$ be arbitrary FIOs given by the integrals \eqref{integral} with the phase functions $\varphi^{(j)}$ and $\varphi$, and let $p_{0,j}$ and $p_0$ be the leading homogeneous terms of corresponding amplitudes. In view of \eqref{sigma-theta} and \eqref{thetas-1}, we have 
\begin{equation}\label{thetas-2}
i^{\Theta_\Phi}\sigma_V\ =\ p_0\quad\text{and}\quad  i^{\Theta_{\Phi_1}}i^{-\Theta_{\Phi_2}}\,\sigma_{V_1}\,\overline{\sigma_{V_2}}\ =\ p_{0,1}\,\overline{p_{0,2}}
\end{equation}
on a set of full measure in $\mathrm{Fix}(\Phi)$. In particular, this shows that the restrictions of $p_0$ and $p_{0,1}\,\overline{p_{0,2}}$ to this set do not depend on the choice of the phase functions. 

Assume now that the FIOs $V_j$ have sufficiently small conic supports, the phase functions $\varphi^{(j)}$ are given by \eqref{phi}, and $V=V_2^*V_1$. Then, as was shown in Subsection \ref{s:a2}, $V$ is represented by the integral \eqref{composition4} with the phase function \eqref{composition3} and an amplitude  with leading homogeneous term $p_{1,0}(y,\eta)\,\overline{p_{2,0}\left(\Phi(y,\eta)\right)}$.
From here and \eqref{thetas-2} it follows that $i^{\Theta_\Phi}\sigma_V=i^{\Theta_{\Phi_1}}i^{-\Theta_{\Phi_2}}\,\sigma_{V_1}\overline{\sigma_{V_2}}$ on a set of full measure in $\mathrm{Fix}(\Phi)$.

\section{ An example where QE holds 
  by~Yves~Colin~de~Verdi\`ere}

\subsection{Introduction}
The goal of this Appendix is to provide an example
of QE for a discontinuous Riemannian metric.
We will first discuss generalities on ergodicity for  Markov
semi-groups in order to get a simple test of ergodicity via the
Poincar\'e maps associated to the singular manifold.
We will then present  examples of metrics on the $2-$sphere whose
geodesic flows are ergodic: these metrics are obtained by glueing together two copies of
an Euclidian disk of radius $1$ via a generic diffeomorphism.  Finally, we will discuss Thom-Sard
transversality in order to check that we have QE for a generic subset
of these metrics. 

\subsection{Ergodicity of Markov semi-groups}
We will denote by $Z$ the unit cotangent bundle of $(X,g)$
and by $d\omega$ the normalized  Liouville measure (a probability)  on $Z$.
We will denote by $G_t$ the semi-group
denoted by $\Xi _t^c$ in the body of the present paper
(see Definition \ref{def:transfer-op})
It acts on $L^1 (Z, d\omega)$ and $L^\infty (Z,d\omega)$ with norms equal to
$1$. 
We call $G_t$ the ``geodesic flow'' even if it is not a 
(deterministic) flow on $Z$.

\begin{definition} The geodesic flow $G_t$   is {\rm ergodic} if and only
  if
the only functions in $L^\infty (Z,d\omega) $ 
 which are invariant under the semi-group $(G_t)_{t\geq
  0}$
are the  functions which are constant outside a set of measure $0$.
 \end{definition}
As a Corollary of ergodicity, we get the
\begin{pro}
If the geodesic flow is ergodic and $f\in L^1(Z, d\omega)$, we have,
 for almost all
$z \in Z$ and in $ L^1(Z, d\omega)$,
\[ \lim _{T\rightarrow +\infty}\frac{1}{T}\int_0^T G_tf(z)dt  = \int _{Z}
f d\omega ~.\]
\end{pro}
This is proved using 
the point-wise ergodic Theorem given in \cite{DS58}, Theorem 5, page 690
or in \cite{Kr85}, Theorem 3.7, page 217: 
\begin{theorem} \label{theo:DS}
 If $\mu $
is a probability measure 
on a space $Z$ and  $(G_t)_{t\geq 0}$ is a strongly measurable semi-group on
  $L^1(\mu)$ whose norms on $L^1(\mu) $ and $L^\infty (\mu)$ are bounded by $1$,
then, for $f\in L^1(\mu )$,  the averages
\[ \frac{1}{T}\int_0^T G_t f(z) dt \]
converge for almost all $z$ and in $L^1( \mu)$
 as $T\rightarrow +\infty $. The limit function
$z\rightarrow \overline{f}(z)$  is invariant
under $G_t$ for all $t$.
\end{theorem}

\begin{definition} A measurable subset $A\subset Z$ is {\rm invariant}
by $G_t$ if the characteristic function $\chi _A$ of $A$ satisfies:
for all $t$,
 $G_t \chi _A (z)\equiv \chi _A (z)$ where $\equiv $ means equality almost everywhere.
\end{definition}

\begin{theorem} If $G_t$ is   given
by $G_tf(x)=\int_Z p_t (x,dy)f(y)$ with $p_t(x,dy )$ a measurable
family
of probability measures on $Z$ (a Markov kernel),  the function
$\overline{f}$ given in Theorem \ref{theo:DS} is the conditional
expectation
 of $f$ with respect to the 
$\sigma-$field  ${\mathcal I} $ of invariant sets:
\[ \overline{f}={\bf E} (f|{\mathcal I})~.\] \end{theorem}
\begin{proof}
This result is proved in the book \cite{Kr85} (Lemma 3.3) for 
Markov chains.
Let us show how to extend it to the continuous case:
let us denote, for $\tau >0$, by $f_\tau $  the function
$\frac{1}{\tau }\int _0^\tau G_t f dt $. We have
clearly
\[ \overline{f}=\lim _{N\rightarrow \infty }\sum _{n=0}^{N-1}
P_\tau ^n f_\tau ~.\]
Hence, using the Markov chain generated by $P_\tau $, we
deduce that $\overline{f}$ is measurable with respect to 
${\mathcal I}_\tau $, the $\sigma-$field of $P_\tau $-invariant sets.
Hence $\overline{f}$ is measurable with respect to 
${\mathcal I}$ which is the intersection of the   $\sigma-$fields 
${\mathcal I}_\tau $. The conclusion follows. 
\end{proof}

Hence, we get:
\begin{cor} The geodesic flow  $G_t$ is ergodic if and only if 
all invariant subsets of $Z$ have measure $0$ or $1$.
\end{cor}

\subsection{Poincar\'e maps}\label{sec:Poincare}
We denote $f\equiv g$ if the functions $f$ and $g$ coincide outside a set of measure $0$
and similarly, for subsets $A$ and $B$, $A\equiv B$ if $A\Delta B$ is  of measure $0$.
Let us assume for simplicity that $ Y$ is smooth and cuts $X$ into two
open 
disjoint parts $X\setminus Y=X_+ \cup X_-$.
Let us denote by $Y_\pm  $ the unit ball bundles for $g_{\pm }^\star$ in
$ T^\star Y$, by ${\mathcal Y}=Y_+ \cup Y_-$. Let  $Z=Z_+ \cup Z_- $ be
the decomposition of the unit cotangent bundle of $X$ into the parts lying over 
$X_+ $ and $X_-$, and let $g_t $ be the classical  geodesic flow on $Z$.
If  $z\in \tilde{Z}_+ $, where  $\tilde{Z}_+ $ is the set of the $z\in Z_+$ such that the geodesic
$t\rightarrow  g_t(z)$, starting from $z$,  cuts $Y$ transversally 
for some $>0$ time, we denote   by 
$\tau (z)$ the first hitting time.
We define a map $\pi_+:\tilde{Z}_+ \rightarrow Y_+$
as follows: if $g_{\tau(z)}=(y,\xi)$,
$ \pi_+ (z)=(y, \xi_{|T_yY })$. We define in a similar way a map $\pi_-$.
We will need the following lemma:
\begin{lem} \label{lemm:inv}
 If $A_+ \subset \tilde{Z}_+$ satisfies $\chi_{A_+}(g_t (z))\equiv \chi _{A_+}(z)$ for all $t$ small enough (depending
on $z$!),
then there exists $B_+ \subset Y_+$ such that $A_+ \equiv \pi _+^{-1}(B_+)$.
\end{lem}
{\it Proof:} If $V$ is the vector field which is
the generator of the classical geodesic flow
in $\tilde{Z}_+ $, the assumption implies that  $V(\chi_{A_+})=0$
in the sense of Schwartz distributions.
Since $V$ is a non-vanishing vector field on a manifold, there are local
coordinates $(x_1,x')$  such that $V=\partial_{x_1}$
and the solutions of $\partial T/\partial{x_1}=0$ in the Schwartz space
are the distributions (here functions) $T$ depending only on $x'$.
 Hence the set $A_+ $ 
is the union of all  maximal 
geodesics in $\tilde{Z}_+ $. These geodesics
are inverse images of points in $Y_+$ by $\pi_+$. 
\qed

The Poincar\'e maps $P_\pm :Y_\pm \rightarrow Y_\pm $ are defined as follows:
if $(y,\eta) \in Y_+$, $P_+ (y,\eta)=\pi_+ (y,\xi_+)$ where $\xi_+$ is the lift of $\eta $
to $Z_+$ pointing in the direction of $X_+$. Similarly for $P_-$. 
The Poincar\'e maps $P_\pm $ are not everywhere defined respectively on $Y_\pm $, but almost
everywhere thanks to the Poincar\'e recurrence theorem.

\begin{definition}
We will say that a set $B \subset Y_+ \cup Y_-$ is invariant under $P_+$
and $P_-$ if for any $y \in Y_+$, $P_+ y \in B$ and similarly for
$y\in Y_-$.\end{definition}

We will use the following result in order to prove ergodicity of the
geodesic flow:
\begin{theorem}Assuming  that 
\begin{itemize}
\item almost all maximal geodesic arcs in $X_\pm $ cut $Y$
  at both ends
\item if $B\subset {\mathcal Y}= Y_+ \cup Y_- $ is invariant under $P_+ $
and $P_-$, then either $B$ or ${\mathcal Y}\setminus B$ are of measure $0$. 
\end{itemize}
the geodesic flow $(G_t)_{t\geq 0 }$ is an ergodic Markov process. 
\end{theorem}
\begin{proof}
Let us assume that the assumptions are satisfied and consider a 
 set $A\subset Z$ invariant under $G_t$. $ A $ is, up to a set of measure $0$, a union of
maximal
geodesic arcs in $X\setminus Y$. Let us consider the traces $B_\pm =\pi_\pm (A_\pm )$
 of $A$ in
${\mathcal Y}$. Let us show that  $B=B_+ \cup B_-$ is  invariant under $P_+$ and $P_-$ modulo sets
of measure $0$.
Let us denote by $w^+_+:Y_+\rightarrow ]0,1]$ and $w^+_-Y_+ \rightarrow [0,1[ $ be the coefficients
 whose sum is $1$
 describing the repartition of the energies after the hitting of the geodesic coming from $Z_+$
and define in a similar way coefficients $w_\pm^- :Y_- \rightarrow [0,1]$.
From the invariance of $A$ under $G_t$ and Lemma \ref{lemm:inv}, we get the  equations
\[ \chi _{B_+}(\zeta) \equiv w^+_+(\zeta) \chi _{B_+}(P_+\zeta )+ w^+_-(\zeta) \chi _{B_-}(P_-\zeta)~,\]
and 
\[  \chi _{B_-}(\zeta) \equiv w^-_+(\zeta) \chi _{B_+}(P_+\zeta )+ w^-_-(\zeta) \chi _{B_-}(P_-\zeta)~.\]
If $\zeta \in Y_+$, $w^\pm _+ (\zeta) \ne 0$ and similarly for 
$w_-^\pm $. We deduce from the fact that characteristic functions take only $0$ and $1$ as values,
that $P_+ (B_+) \subset B_+ $ and $P_- (B_+) \subset B_- $  modulo sets of measure $0$ and 
similarly for $B_-$.
  If  $B_1$ is the intersections of the
images of $B$ by iterating $P_+$ and $P_-$. The measure of $B\setminus
B_1$ vanishes. The second  assumption  implies that either $B_1$ or
$Y\setminus B_1$ has measure $0$.
From Lemma \ref{lemm:inv}, we have that
$A_\pm \equiv \pi_\pm ^{-1}(B_\pm )$ and then the first assumption finishes  the proof. 

\end{proof}

\subsection{An example: Gluing together two flat disks}
\begin{figure}[hbtp]
\leavevmode \center
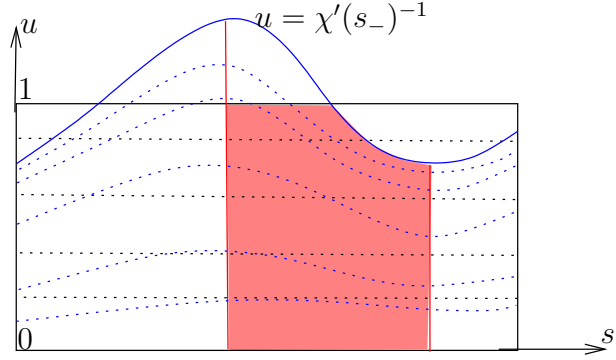
\caption{{\it \small Poincar\'e section.}}
 \label{fig:poincare}
\end{figure}

Let us consider two unit Euclidian disks $D_+ $ and $D_-$
and a diffeomorphism $\chi :\partial D_- \rightarrow \partial D_+$ so that $\chi''(s)\ne
1$
except at a finite number of values of $s$.
Gluing together $D_+ $ and $D_-$ along their boundaries using
$\chi $  gives a topological manifold homeomorphic to $S^2$
with a metric $ g_\chi $ which is flat outside the equator and discontinuous on
the equator except at a finite number of points.

Let us first describe the Poincar\'e section:
as in Section \ref{sec:Poincare}, we define $Y_\pm \subset T^\star \partial
D_+$ with the boundary $\partial D_+=\field{R}/2\pi \field{Z}$   parametrized by the arc length
$s$.
Then $Y_+ =\{ (s,u)|s\in \field{R}/2\pi \field{Z} , |u|<1 \} $ with the symplectic structure $\omega=
du\wedge ds $.
Using the map $\chi $ and its extension $\Xi $ to the cotangent bundle of
$D_-$,
we get $Y_- =\Xi ( \field{R}/2\pi \field{Z} \times ]-1,+1[)$ or more
explicitly 
\[ Y_- = \{ (s,u)|s\in\field{R}/2\pi \field{Z} , |u|< \psi (s) \} \]
with $\psi (s)= 1/\chi'(\chi^{-1}(s))$. 

The Poincar\'e maps $P_\pm : Y_\pm \rightarrow Y_\pm $ are integrable; they
preserve the foliations ${\mathcal F}_\pm $ of $Y_\pm $  defined
by
\begin{itemize}
\item 
 ${\mathcal F}_+ =\{ L^+_\alpha |~ |\alpha | <1 \}$
with $  L^+_\alpha:=\{ (s,\alpha)~|~s\in\rbb/2\pi \zbb  \} $ on which $P_+ $ acts as a
rotation
of angle $\rho _+ (\alpha )=2 \arccos (\alpha ) $ satisfying  $\rho_+'<0 $
\item  ${\mathcal F}_- =\{ L^-_\beta |~ |\beta | <1 \}$
with $  L^-_\beta:=  \{ (s,\beta \psi(s))~|~s\in\rbb/2\pi \zbb \} $ on which $P_- $ acts as a
rotation
of angle $\rho _- (\beta ) $ where $\rho_-'<0 $. 
\end{itemize}
We have $\psi'(s)= \eta(s)  \chi'' (\chi^{-1}(s)))  $ where $\eta $
does not vanish. This implies that the foliations ${\mathcal F}_+ $
and ${\mathcal F}_- $ are transverse in $Y_+ \cap Y_-$  outside a finite
number of segments $I_j:= \{s_j \} \times ] - \min (1,\psi(s)), \min
(1,\psi(s))[ $ with $0\leq s_1 < s_2 < \cdots <s_N <2\pi $.

Our main result is:
\begin{theorem} The geodesic flow on the 2-sphere $(S^2,g_\chi ) $ has
  two 
  ergodic components corresponding in the Poincar\'e sections to $u>0 $
  and to $u<0$.
 \end{theorem}
\begin{proof}
The idea of the proof is that the refractions will mix the two foliations in an erratic
way depending on the derivative of $\chi $ at the impact point.

Following the result of Section \ref{sec:Poincare}, we have to
consider a subset $B $ of $Y=(Y_+ \cup Y_-) \cap \{ u>0\} $ which is invariant under
$P_+$
and by $P_-$, meaning that $P_+ (B \cap Y_+)= B  \cap Y_+ $ and 
 $P_- (B \cap Y_-)=B \cap Y_- $.  
We want to
prove that $B$ or $Y\setminus B$ has measure $0$. 
Let $B_i$ be the intersection of $B$ with the leaves $L_+^\alpha $  
on which the rotation $\rho_+(\alpha)/2\pi  $ is irrational. Then $B_i
\equiv B$ (i.e. $B\setminus B_i$ has measure $0$) and 
$B_i   \cap  L_+^\alpha $   is measurable and invariant under the rotation
$\rho_+ (\alpha)$. Hence the measure of $B_i   \cap  L_+^\alpha $ is $0 $ or
$2\pi $ by the ergodicity of the irrational rotations of the circle. 
From this we get that $B\cap Y_+ $ is equivalent to a set foliated by ${\mathcal
  F}_+ $. Similarly $B\cap Y_-$ is equivalent to a set foliated by
${\mathcal F}_-$. 
Let us now consider the set $B_j:=B\cap D_j $ with
$D_j:= \{ (s,u)|~ s_j<s <s_{j+1}, 0<u <  \max
(1,\psi(s))  $.  In $D_j$, both foliations are transverse. This
implies  that $B_j$ or $ D_j \setminus B_j $ is of
measure $0$: using smooth coordinates $(x,y)$ in $D_j$ so that
the two foliations correspond to $x={\rm const} $
and $y={\rm const} $ respectively, the indicator function of
$B$ is equivalent to a function depending on $x$ only  and 
to a function depending of $y$ only, hence is equivalent to a constant
$0$ or $1$.
 If $B_j $ is of measure $0$,  then $B\cap Y_+ $ is of
measure $0$ as being foliated by ${\mathcal F}_+$, similarly for $B\cap
Y_-$.
The conclusion follows: all $B_j$ are of measure $0$ or all $D_j
\setminus B_j$ are of measure $0$. 
\end{proof}

If we want to apply the main Theorem of the paper, we have to take into
account the fact that there are two ergodic components, they are
equivalent by the involution $J: (x,\xi)\rightarrow (x,-\xi )$  which on the
quantum level is the complex conjugation $ \phi \rightarrow
\overline{\phi}$. The semi-classical measures associated to real
eigenfunctions are  invariant under $J$. 

We want to prove that QE holds for the manifold
$(S^2,g_\chi)$  for a generic $\chi $.
What remains to do is to check that $\Xi_t^c$ coincides with
$\Xi_t^d$ for a generic $\chi $: this is a consequence of the fact 
that Cauchy data of  recombining geodesics have measure $0$ and will
be proved in what follows to hold for a generic $\chi $.
We recall the 
\begin{definition}
 We denote by $ f \pitchfork  Z$ the fact that the map
$f:X\rightarrow Y $  is
transverse to the sub-manifold  $Z$ of $Y$, i.e. for each $x\in X$ so that
$z=f(x) \in Z$, we have  $T_{z}Y =f'(x)\left(T_x X \right) + T_{z} Z$.
\end{definition}

We start with the following lemma:
\begin{lem}  Let us consider a  word $P_\chi^\alpha =P_-^{a_1} P_+^{a_2} 
\cdots P_+^{a_{2l}} $ with $a_j\in \field{Z} \setminus 0$. For $N\geq 3$
and $D\subset Y_+$ a closed domain with a smooth boundary,   so that $D\subset Y_+$, 
let  
${\mathcal A}^N_D $ be  the manifold  of all  diffeomorphisms of class $C^N$
of $S^1$ so that $P^\alpha _\chi $ is defined
in some open set $V$ containing  $D$.
 Let us denote by $\pi $ the projection of $Y_+ $ onto
$S^1$, by $W$  the diagonal of $S^1 \times S^1$  and, for
$\chi \in {\mathcal A}^N_D $,  by 
$\rho (\chi )$ the $C^2$ map from  $D $ into
$S^1\times S^1$ defined by
\[ \rho (\chi )(z)=\left( \pi (z),\pi \left(P^\alpha _\chi
    (z)\right)\right)~.\]
Then the set of $\chi$'s belonging to ${\mathcal A}^N_D $ so that $\rho (\chi ) \pitchfork W  $ 
is open and dense in ${\mathcal A}^N $.
\end{lem}
\begin{proof}
By induction on $|\alpha|$, we can assume that we are looking 
only at the case where the projections on $S^1$  of the points
of the $z-$orbit  $(z,P_+ z=z_1, \cdots , z_{|\alpha|-1} )$ are pairwise distinct.

The openness is clear.

The density follows from the transversality Theorem as  stated for example
 in \cite{AB63} and \cite{AR67}, page 48 (see Appendix \ref{app:transv}).
We will apply Theorem \ref{theo:trans} with $r=2$, 
$X =D $, $Y=  S^1 \times  S^1 $ and $W$ the diagonal of $Y$.

The transversal intersection of $\rho ( \chi)$ with $W$ implies that the
set of $z$ for which $\pi (z)=\pi \left(P^\alpha _\chi (z)\right)$
is a submanifold of dimension $1$ of $Y_+$.
Let us consider the evaluation map
${\rm ev}(\chi, z)=(\pi(z),\pi \left( P_\chi ^\alpha (z)\right)$.
The differential $L$  of ${\rm ev}$ at a point $(\chi_0, z_0)$ can be
written
as $L(\delta \chi ,\delta z)= (0, L_1 \delta \chi)  + (\delta s, L_2 \delta z)$.
In order to prove transversality it is enough to prove 
that $L_1$ is surjective. Let us restrict ourselves to variations of
$\chi $ in some small neighborhood of $s_1=\chi ^{-1}(s_0)$ where
$z_0=(s_0,u_0)$. Then we have 
$L_1 (\delta \chi )= \delta \chi (s_1)$.  
\end{proof}

Hence we get that
\begin{pro} For any $N \geq 3$ , the set of $C^N$
  diffeomorphisms
$\chi$'s, whose  set
  of  periodic   points under iterations of $P_+,P_-$ and their inverses  is of measure $0$, 
is generic.
\end{pro}

This  implies that $\Xi_t^c $ and $\Xi _t^d$ coincide 
for a generic $\chi$. Hence, a small extension of
Theorem \ref{theo:main} applies and we get
\begin{theorem}For a generic $\chi $,   
 any  basis of real eigenfunctions of $\Delta _\chi $
 is  QE. \end{theorem}

\begin{rem} Unique Quantum Ergodicity is not satisfied because there
  are infinitely many radial eigenfunctions corresponding to the
  Neumann and the Dirichlet problem for radial functions in the unit
  Euclidian disk. 
\end{rem}

\subsection{The Abraham-Thom transversality Theorem}
 \label{app:transv}

Let us give the statement of the transversality Theorem, due to Ren\'e
Thom,  as given in
\cite{T56,AB63,AR67}:
\begin{theorem} \label{theo:trans} Let $r\geq 1$ and 
 ${\mathcal A},~ X$ and $Y$ be $C^r$ manifolds. We assume that
 ${\mathcal A}$ is a Banach manifold while $\dim X $ and $\dim Y $
are finite. The manifold $X$ is assumed to be compact with  a smooth boundary.
 Consider a $C^r$ map
$\rho : {\mathcal A} \rightarrow C^r (X,Y)$
and $W\subset Y$ a compact  sub-manifold. The
evaluation
map ${\rm ev}:{\mathcal A}\times  X\rightarrow Y $ is defined by ${\rm
  ev}(a,x)=\rho(a)(x)$ and we denote by ${\mathcal A}_W$   the set of the
$a $'s 
in $  {\mathcal A}$ so that 
$ \rho (a) \pitchfork W $.
Then if $r> \max (0, \dim X -{\rm codim} W )$,
and ${\rm ev}\pitchfork W $, then ${\mathcal A}_W $ is 
 open  and dense in  ${\mathcal A}$.
\end{theorem}

\vskip2cm

\noindent{\bf Acknowledgements.}
The authors would like to thank Steve Zelditch for useful discussions and comments. D.J. and Y.S. thank
the Fields Institute in Toronto and all of the authors are grateful to the CRM
in Montreal for the hospitality. 
Y.CdV thanks S\'ebastien Gou\"ezel for illuminating
discussions on ergodic theorems.

\end{document}